\documentclass[twoside,11pt]{article}
\usepackage{amsmath}\usepackage{amsfonts}
\usepackage{amsthm}\usepackage{graphicx}
\usepackage{graphicx}
\usepackage[pagewise]{lineno}
\usepackage{color}
 \def\R{\mathbb{R}}
\newcommand{\D}{\displaystyle}
by-\headheight \advance \topmargin by-\headsep 
\usepackage[numbers,sort&compress]{natbib}
\newtheorem{thm}{Theorem}[section]

\newtheorem{lemma}[thm]{Lemma}

\newcommand{\cred}{\color{black}}
\begin{document}
\title{\bf A kinetic chemotaxis model with internal states and temporal sensing}
\author{Zhi-An Wang \thanks{Department of Applied Mathematics, The Hong Kong Polytechnic University, Hung Hom, Kowloon, Hong Kong ({\tt mawza@polyu.edu.hk}). }}
\date{}
\maketitle
%\begin{center}
%{\it Dedicated to Professor Glenn Webb on the occasion of his $70$th birthday}
%\end{center}
\begin{quote}
\textbf{Abstract}: By employing the Fourier transform to derive key {\it a priori} estimates for the temporal gradient of the chemical signal, we establish the existence of global solutions and hydrodynamic limit of a chemotactic kinetic model with internal states and temporal gradient {  in one dimension}, which is a system of two transport equations coupled to a parabolic equation proposed in \cite{DS}.


\end{quote}

\section{Introduction}
\setcounter{equation}{0}
\renewcommand{\theequation}{\thesection.\arabic{equation}}
%Chemotaxis, the directed migration of cells in response to external chemical cues (signals), is a fundamental cellular process and
%essential for development, tissue hemostasis, wound healing, immune response and progression of many diseases such as angiogenesis finding food, repellent action and forming the multi-cellular body of protozoa.
The mathematical models of chemotaxis were generally constructed at two scales of interest:  population (macroscopic) or cellular (microscopic) scale.
The prototype of the population-based chemotaxis model was  proposed by Keller-Segel in the 1970s \cite{KS1} to describe the aggregation of cellular slime molds {\it Dictyostelium discoideum} in response to the chemical cyclic adenosine monophosphate (cAMP). The first microscopic description of chemotaxis model was due to Patlak \cite{Patlak} where the kinetic
theory was used to express the chemotactic velocity in term of the
average of the velocities and run times of individual cells. This approach was essentially developed in  \cite{Stroock1974, Alt1, ODA} using a
velocity-jump process assuming that cells run with some velocity
and at random instants of time they changes velocities (directions)
according to a Poisson process with the intensity $\lambda$.  The
governing evolution equation for the simplest version of this
process reads
\begin{equation}\label{eq1a3}
\frac{\partial }{\partial t}p(t,x,v)+v\cdot \nabla p(t,x,v)=-\lambda
p(t,x,v)+\lambda \int_V \mathcal{K}[S](v',v)p(t,x,v')dv',
\end{equation}
where $p(t,x,v)$ denotes the density of particles at position $x \in
\R^N (N\geq 1)$, moving with velocity {  $v \in V$} at time $t \geq
0$ \cite{ODA} and $V$ is a symmetric compact set in  $\R^N$.  Here $\lambda$ is called the turning frequency and
$1/\lambda$ is a measure of the mean run length between velocity
jumps. The kernel function $\mathcal{K}[S](v',v)$ is the density
distribution function of a velocity jump from $v'$ to $v$ if a jump
occurs, which is a function of the chemical concentration $S(t,x)$. Generally speaking both
turning frequency $\lambda$ and turning kernel may also depend on
internal or external variables.

The  microscopic models of chemotaxis can incorporate the individual cell properties which may be passed to the macroscopic models via appropriate spatial/temporal scalings. When the kernel function $\mathcal{K}[S]$ depends on the chemical concentration or its spatial derivative and $S(t,x)$ satisfies some appropriate dynamical equation,  the kinetic system \eqref{eq1a3} has been extensively studied in the literature with focuses on the global well-posedness (cf. \cite{CMPS,EH,HKS1,HKS2}) and macroscopic limits (see \cite{OH,EO1,Xue1,STY, PTV} for formal derivation and \cite{CMPS,HKS1,HKS2,JV, Liao, ENV} for rigorous justification). In this paper, we consider following kinetic equation with internal dynamics proposed in  \cite{ODA}
\begin{equation}\label{eq1a8}
\frac{\partial p}{\partial t}+v\cdot \nabla p+\nabla_{\xi}\cdot(\eta
p)=-\lambda(S) p(t,x,v,{\cred \xi})+\int_V \lambda(S)\mathcal{K}[S](v',v)p(t,x,v',{\cred \xi})dv',
\end{equation}
where $\xi \in Z \subset \R^m (m\geq 1)$ denotes the internal variable
which evolves according to the equation
\begin{equation*}\label{eq1a4}
\frac{d\xi}{dt}=\eta(\xi,S(t,x(t)))
\end{equation*}
with $\eta(\cdot,S): Z \to \R$ being a function describing the signal transduction and $S(t,x)$ denoting the
concentration of chemical signal and $x(t)$ is the cell moving path. Here the internal dynamics of cells is included through {\cred $\xi$} and the chemical signal is incorporated into the turning frequency. The kernel function is a non-negative and satisfies the normalization condition
\begin{equation}\label{eq1a9}
\int_V \mathcal{K}[S](v',v)dv=1.
\end{equation}
To write (\ref{eq1a8}) in a compact form, we introduce a notation
\begin{equation}\label{eq1a10}
T[S](v',v)=\lambda(S)\mathcal{K}[S](v',v)
\end{equation}
which is called the turning kernel. Then equation (\ref{eq1a8}) can be rewritten as a compact form
\begin{equation}\label{eq1a11}
\frac{\partial p}{\partial t}+v\cdot \nabla p+\nabla_{\xi}(\eta
p)=\int_V (T[S]p'-T^*[S]p)dv',
\end{equation}
where the abbreviation $p'=p(t,x,v',\xi), T[S]=T[S](v',v)$, $T^*[S]=T[S](v,v')$ and the
intensity $\lambda$ of the Poisson process is thus given by from (\ref{eq1a9})-(\ref{eq1a10})
\begin{equation*}
\lambda(S)= \int_V T^*[S]dv'.
\end{equation*}
Then (\ref{eq1a11}) becomes an equation same as the one considered in
paper \cite{DS}, {\cred where the turning kernel $T[S]$ is assumed to be independent of $\xi$ (i.e. $\xi$ has no influence on cell movement). When $T[S]$ replies on the internal variable $\xi$, there are some results available as follows}.  When the signal response function $\eta$ has some stiffness, the macroscopic equation of Keller-Segel type as a parabolic limit of \eqref{eq1a11} was derived in \cite{PTV} and the global existence of solutions of \eqref{eq1a11} coupled to an elliptic equation for the chemical signal $S$ was proved in \cite{Liao}. The formal macroscopic limit of \eqref{eq1a11} with certain specific internal dynamics was previously derived in \cite{EO1, EO2}. In this paper, we shall consider another approach illustrated in \cite{DS} by considering the averaging effect of {  {  internal} dynamics to derive the dynamics of internal variable instead of a given dynamics as recalled above, and investigate the global existence of solutions to the resulting equations. To make our presentation self-contained, we shall briefly recall some derivations shown in \cite{DS} below.

Define $f$ and $\rho$ as
\begin{equation*}\label{eq1a16}
f(t,x,v)=\int_Z p(t,x,v,\xi)d\xi, \ \rho(t,x)=\int_V f(t,x,v) dv.
\end{equation*}
That is $f$ is the position-velocity density of cells and $\rho$ is the total density of cells over all velocities.
Then the average value of the internal variables is defined by
\begin{equation*}\label{eq1a12}
z(t,x)=\frac{1}{\rho} \int_V \int_Z \xi p(t,x,v,\xi)d\xi dv.
\end{equation*}
Here we assume the boundary condition $p(t,x,v,\xi)=0$ for $\xi \in
\partial Z$ {\cred where $\partial Z$ denotes the boundary  of $Z$}, and two moment closure assumptions
\begin{equation}\label{eq1a13}
\int_V\int_Z \xi v pd\xi dv=z \int_V vfdv, \ \int_V\int_Z \eta pd\xi
dv=\rho \bar{\eta}(z,S).
\end{equation}
{\cred The first closure assumption in \eqref{eq1a13} means that the variables $v$ and $\xi$ are uncorrelated (namely the internal variable $\xi$ has no influence on cell movement) and it can be fulfilled, say, for $p(t,x,v,\xi)=f(t,x,v)\tilde{p}(t,x,\xi)$. The second closure assumption in \eqref{eq1a13} depends on the form of {${\eta}(z,S)$} (an example will be discussed in section 2).} Upon an integration of  (\ref{eq1a11}), the following equations are obtained
\begin{eqnarray}
&\D f_t+ v\cdot \nabla f=\int_V
(T[S]f'-T^*[S]f)dv',\label{eq1a14}\\
&\D (\rho z)_t +\nabla \cdot \bigg(z \int_V v fdv \bigg)=\rho
\bar{\eta} (z,S). \label{eq1a15}
\end{eqnarray}
{\cred More detailed explanations of \eqref{eq1a13} and derivation of \eqref{eq1a14}-\eqref{eq1a15} are referred to \cite{DS}}.   {We remark here that \eqref{eq1a14} is weakly coupled to \eqref{eq1a15} in the sense that $S$ depends on the internal variable $z$ (see \eqref{eq2a9} in section 2).} A typical form of the turning kernel $T$ {  depending} on spatial-temporal gradient of the chemical signal is  (cf. \cite{DS})
\begin{equation}\label{eq1a17}
T[S]=\varphi(S_t+v\cdot \nabla S),
\end{equation}
where $\varphi: \R \to \R$ is a smooth monotonically decreasing function.

For given chemical concentration $S(t,x)$, when the turning kernel function in \eqref{eq1a17} satisfies
\begin{equation}\label{eq1a18}
0<\alpha \leq \varphi \leq \beta
\end{equation}
with two constants $\alpha, \beta >0$, the hydrodynamic limit of equations (\ref{eq1a14})-(\ref{eq1a15}) was derived in \cite{DS}. {\color{black} When (\ref{eq1a14}) is coupled to a reaction-diffusion equation for $S$
\begin{equation}\label{eqnS}
\tau S_t=\Delta S+\rho-S
\end{equation}
with $\tau=\{0,1\}$, and the turning kernel $T[S]$ depends on $S_t$ implicitly (meaning that the bound of $T[S]$ does not depend on  $|S_t|$), the global existence and parabolic limit of solutions were obtained in a series of works \cite{HKS2, HKS1, HKS3} for $x\in \R^N (1\leq N\leq 3$). When $\varphi$ in \eqref{eq1a17} satisfies $\varphi \in C^1(\R)\cap L^\infty(\R)$ and $\varphi'<0$, {\cred the global existence and numerical simulation of solutions were investigated in \cite{NV1}}, and the so-called flux-limited Keller-Segel model was formally derived in \cite{SCB-PNAS} and justified in \cite{PVW} where the  global dynamics of the resulting  flux-limited Keller-Segel model was also studied. However, when $T[S]$ explicitly depends on $|S_t|$ (i.e. the bound of $T[S]$ depends on $|S_t|$), the only result available so far is the global existence of solutions to (\ref{eq1a14}) coupled with \eqref{eqnS} with $\tau=0$ in $\R$ established in \cite{EH}.  For the kinetic chemotaxis model with internal state and turning kernel explicitly depending on the temporal gradient $S_t$, no results seem to be available as  we know.   The purpose of this paper is two folds: (1) establish the existence of global solutions of the model (\ref{eq1a14})-(\ref{eq1a15}) in $\R$ with a more general turning kernel explicitly depending on the temporal gradient $S_t$ where $S$ satisfies a reaction-diffusion equation derived in \cite{DS} (see section 2); (2) relax the condition \eqref{eq1a18} imposed in \cite{DS} and rigorously derive the hydrodynamic limit of (\ref{eq1a14})-(\ref{eq1a15}).}

The rest of this paper is organized as follows. In section 2, we briefly recall the derivation of  the equation for the chemical concentration $S$  shown in \cite{DS, EO1} and state our main results. In section 3, we present some preliminary results. In section 4, we show the local and global existence of solutions.  Finally we discuss the hydrodynamic limit of the model in section 5.

\section{Model review and main results}
\setcounter{equation}{0}
\renewcommand{\theequation}{\thesection.\arabic{equation}}
To complete the statement of our problem, we briefly review the function $\eta(z,S)$ and  dynamical equation for the chemical signal $S$ derived in \cite{DS,EO1}. Then we propose the condition on the turning kernel $T[S]$ and state our main result.

In general the dynamics of the chemical concentration $S$  follows a reaction-diffusion {  equation}
\begin{equation}\label{eq2a1}
\frac{\partial S}{\partial t}=\Delta S+\phi(S, z, \rho),
\end{equation}
where  $\phi(S,z,\rho)$ describes the production and degradation of
the chemical signal in dependence on the cell density and the internal
state of the cells.

For internal dynamics, we employ the model studied in \cite{DS, EO1, EO2}. This model assumes that the
internal states $\xi=(\xi_1, \xi_2)$ satisfies a Cartoon model
\begin{equation}\label{eq1a6}
\frac{d\xi_1}{d\tau}=\frac{g(S(t,x))-(\xi_1+\xi_2)}{\tau_e},\
\frac{d\xi_2}{d\tau}=\frac{g(S(t,x))-\xi_2}{\tau_a},
\end{equation}
where $\tau_{e}$ and $\tau_{a}$ are respective excitation and adaptation time scales in the signal transduction,
and the production of the chemical signal is triggered by the first internal variable. Then the reaction term $\phi(S,z,\rho)$ takes the form
\begin{equation}\label{eq2a2}
\phi(S,z,\rho)=\rho w-\Psi(S),
\end{equation}
where $w$ denotes the concentration of the first state $\xi_1$ and
the function $-\Psi(S)$ models the degradation of $S$, which has the form
after rescaling (see \cite{DS})
\begin{equation}\label{eq2a3}
\Psi(S)=S(1+S).
\end{equation}
In general it is reasonable  (e.g. see \cite{EO1}) to assume that $\tau_e <<\tau_a$ in (\ref{eq1a6}) due
to the fast excitation and slow adaptation of signaling process. Effectively we may assume that $\tau_e=0$.
Then from the first equation of (\ref{eq1a6}), we have that
\begin{equation*}
\xi_1=g(S)-\xi_2,
\end{equation*}
which implies {  that} the concentration of the first internal variable can
be represented by the concentration of the second internal variable.
Therefore we  focus on the second internal chemical $\xi_2$ and
hence $z \approx \xi_2$. Consequently from the moment closure
assumption (\ref{eq1a13}) and internal dynamic equation
(\ref{eq1a6}), we deduce that
\begin{equation}\label{eq2a4}
\bar{\eta}(z,S)=\eta(z,S)=\frac{{\cred g(S(t,x))}-z}{\tau_a}
\end{equation}
and the concentration $w$ of the first state variable $\xi_1$ is chosen as
\begin{equation*}\label{eq2a5}
w=(g(S)-z)+
\end{equation*}
where the notation $q_+$ denotes the positive part of $q$.
The function $g\geq 0$ describes the mechanism of signal transduction and we
assume that
\begin{equation}\label{eq2a6}
g \in C^1([0,\infty)) \ \text{and} \ g \ \text{is bounded for any}\ S\geq 0.
%g \ \text{is bounded and} \ g'(S) \ \text{is continuous for any}\ S \in \R.
\end{equation}
Typically $g(S)$ vanishes at zero, is monotone increasing and saturates for large $S$.
A suitable representation of $g$ should depend on the fraction of receptors occupied. A prototypical  form is
\begin{equation*}
g(S)=\Phi \bigg( \frac{S}{K_D+S} \bigg)
\end{equation*}
with some continuous function $\Phi$, where $K_D$ denotes the
dissociation rate for the chemical signal \cite{BSB}. A typical
choice (e.g. see \cite{DS}) is that $\Phi={\rm identity}$, i.e.,
$g(S)=\frac{S}{K_D+S}$. In general the turning kernel depends on $S, S_t$ and $\nabla S$ and hence denoted by $T[S, S_t, \nabla S]$.  For brevity, in what follows, we use the short form $T[S]$ to denote $T[S,S_t,\nabla S]$ if there is no confusion caused. {  Without loss of generality we assume $\tau_a=1$ in the sequel.}

Then substituting (\ref{eq2a2}) and (\ref{eq2a4}) into (\ref{eq2a1})
and coupling the resulting equation with (\ref{eq1a14})-(\ref{eq1a15}) lead to the following {  one-dimensional} system
\begin{eqnarray}
%\begin{aligned}
f_t+ v\cdot f_x &=&\int_V (T[S]f'-T^*[S]f)dv',\label{eq2a7}\\
(\rho z)_t + (z j)_x &=&\rho g(S)-\rho z, \label{eq2a8} \\
S_t&=&S_{xx}+g(S)-z\rho-S(1+S) \label{eq2a9}
%\end{aligned}
\end{eqnarray}
for {\color{black} $(x,v) \in \R\times V$ where $V\subset \R$ is a bounded interval} and
\begin{equation}\label{flux}
j(t,x)=\int_V vf(t,x,v)dv
\end{equation}
 denotes the cell density flux.
To complete the statement of the problem, we impose the following initial conditions
\begin{equation}\label{eq2a10}
f|_{t=0}=f_0, \ z|_{t=0}=z_0, \ S|_{t=0}=S_0.
\end{equation}
Due to the biological relevance, we assume $f_0, z_0, S_0$ are all non-negative. The main result of this paper is the global
existence of solutions to the one-dimensional system (\ref{eq2a7})-(\ref{eq2a10}), where the turning kernel fulfills the following structural assumption:

\begin{itemize}
\item[(H)] The turning kernel $T[S]$ is a Lipschitz continuous function satisfying the growth condition
\begin{eqnarray}\label{kernel}
0\leq T[S] \leq C_0(1+\|S\|_{W^{1, \infty}(\R)}+\|S_t\|_{L^\infty(\R)})
\end{eqnarray}
with some constant $C_0>0$.
\end{itemize}
The assumption (H) entails that the boundedness of turning kernel depends on the magnitude of $S, S_x$ and $S_t$, instead of being uniformly  bounded as assumed in \eqref{eq1a18} and in \cite{PVW}. A typical example is $T[S]=\lambda_0+\sigma (S_t+v  S_x)$ with positive constant $\lambda_0$ and non-zero constant $\sigma$, see \cite{Ford:1991:MBR2}.
\bigbreak

The main result of this paper is the global existence of solutions of (\ref{eq2a7})- (\ref{eq2a10}), which is stated in the following theorem.
\begin{thm}\label{mth}
Let the structure condition (H) and assumption (\ref{eq2a6}) hold.
Assume that $f_0 \in L^1\cap L^{\infty}(\R \times V), z_0 \in
L^{\infty}(\R)$ and $S_0\in  W^{2,\infty}(\R)$. Then the system
(\ref{eq2a7})- (\ref{eq2a10}) has a global solution $(f,z,S)$ satisfying for all $ 1\leq p\leq \infty$
\begin{equation}\label{eq4a34}
f\in L^{\infty}([0,\infty);L^p(\R \times V)),
\end{equation}
\begin{equation}\label{eq4a35}
z \in L^{\infty}([0,\infty); L^{\infty}(\R)),
\end{equation}
\begin{equation}\label{eq4a36}
S \in
L^{\infty}([0,\infty); W^{2,\infty}(\R)) \cap C([0,\infty); L^p(\R)).
\end{equation}
such that
$$\|z\|_{L^\infty(\R)}+\|S\|_{L^\infty(\R)}+\|S\|_{W^{1,q}(\R)}\leq C, \ 1\leq q<\infty$$
holds for some constant $C>0$ independent of $t$.
\end{thm}
{\color{black}
The key  to prove the global existence of solutions of (\ref{eq2a7})-(\ref{eq2a10}) is to derive a priori estimates of $S_x$ and $S_t$ so that the bound of $T[S]$ assumed in \eqref{eq2a6} can be controlled. However this is a very challenging issue since there is no general theory or method to estimate the spatial or temporal gradient of solutions for parabolic equations. The main new contribution of this paper is to employ the Fourier transform to derive a priori estimates of $S_x$ and $S_t$ in one dimension (see Lemma \ref{1destimate1} and Lemma \ref{1destimate2}) which enable us to derive the global existence of solutions by using a Gronwall inequality with logarithmic nonlinearity (see Lemma \ref{grownwall}). The results in Lemma \ref{1destimate1} and Lemma \ref{1destimate2} can only be ensured in $\R$ while remain open in the multi-dimensional space $\R^N (N\geq 2)$.

The second result of this paper is to extend the result of hydrodynamic limits of (\ref{eq2a7}) established in \cite{DS} under the restrictive condition \eqref{eq1a18} to a more general turning kernel $T[S]$ which is allowed to be unbounded. The details are presented in section 5.
}

\section{Preliminary results}
\setcounter{equation}{0}
\renewcommand{\theequation}{\thesection.\arabic{equation}}
In this section, we introduce and prove some preliminary results. In what
follows, we use $C$ to denote a positive generic constant which can change
from one line to another. By $\Gamma(t,x)$ we denote the fundamental
solution of the differential operator $\partial_t-\Delta_x+1 $ in
$\R^N (N\geq 1)$
\begin{equation}\label{eq3a2}
\Gamma(t,x)=\frac{1}{(4\pi t)^{N/2}}
\exp\bigg(\frac{-|x|^2}{4t}-  t \bigg).
\end{equation}
%We first have the following results.
%
%\begin{lemma}\label{fundsolu}
%Suppose that $\mathcal{B}(t,x)=e^{-\frac{x^2}{4t}}, x\in \R^N$, then
%for any $1\leq p \leq \infty$, $0\leq l,k\leq \infty$, there is a
%constant $C$ such that for any $t>0$ it holds that
%\begin{eqnarray*}
%\begin{aligned}
%&\|\partial^l_t\partial^k_x \mathcal{B}(t, \cdot)\|_{L^p(\R^N)} &\leq
%Ct^{\frac{N}{2p}-l-\frac{k}{2}},\\
%&\|\partial^l_t\partial^k_x  \Gamma(t,\cdot)\|_{L^p(\R^N)} &\leq C
%e^{-  t}t^{\frac{N}{2p}-\frac{N}{2}-l-\frac{k}{2}}.
%\end{aligned}
%\end{eqnarray*}
%\end{lemma}
%
%{\it Proof}. The first inequality can be easily obtained using
%fundamental calculus and the second inequality follows from the fact that
%$\Gamma(t,x)=\frac{e^{-t}}{(4\pi t)^{N/2}}\mathcal{B}(t,x)$.
%
%\qed

\begin{lemma}\label{fundsolu1}
Let $1\leq p \leq +\infty$ and $\Gamma(t,x)$ be defined in (\ref{eq3a2}). Then
it holds that
\begin{eqnarray*}
%\int_0^{\infty} \|\Gamma(t, \cdot)\|_{L^p(\R^N)}dt<+\infty, \ {\rm if}\ \frac{1}{p}>1-\frac{2}{N},\\
\int_0^{\infty} \|\partial_x^k \Gamma(t, \cdot)\|_{L^p(\R^N)}dt
<+\infty, \ {\rm if}\ \frac{1}{p}>1+\frac{k-2}{N}.
\end{eqnarray*}
\end{lemma}

\begin{proof}First recall a basic result (cf. \cite{PVW}): for any $1\leq p \leq \infty$, $0\leq l,k\leq \infty$, there is a
constant $C$ such that for any $t>0$ it holds that
\begin{eqnarray}\label{Gamma}
\|\partial^l_t\partial^k_x  \Gamma(t,\cdot)\|_{L^p(\R^N)} &\leq C
e^{-  t}t^{\frac{N}{2p}-\frac{N}{2}-l-\frac{k}{2}}.
\end{eqnarray}
Then we immediately have that
\begin{eqnarray*}
\begin{array}{lll}
\D \int_0^{\infty} \|\partial^k_x\Gamma(t)\|_{L^p(\R^N)}dt\leq
  \D C \int_0^{\infty}e^{-  t}
t^{\frac{N}{2}(\frac{1}{p}-1)-\frac{k}{2}}dt.
\end{array}
\end{eqnarray*}
Then the above integral is bounded if
$\frac{N}{2}(\frac{1}{p}-1)-\frac{k}{2}>-1$, i.e., $\frac{1}{p}>1+\frac{k-2}{N}$. The proof is completed.
\end{proof}

Next we present a generalized Gronwall's type inequality for later use based on an inequality in \cite[Lemma 4]{HKS1}. Throughout the paper, we use the notation $\phi_+$ to denote the positive part of $\phi_+$, namely $\phi_+=0$ if $\phi\leq0$ and $\phi_+=\phi$ if $\phi>0$.

\begin{lemma}\label{grownwall}
Let $a(t)$ and $b(t)$ be positive functions. If the function
$y(t)\geq 0$ is differentiable in $t$ and satisfies
%\begin{equation*}
%\frac{d y}{dt} \leq a(t)y |\ln y|+b(t)y,
%\end{equation*}
%or
\begin{equation}\label{eq}
y(t)\leq y(0)+\int_0^t {  [}a(\tau)y(\tau)(\ln
y(\tau))_++b(\tau)y(\tau){ ]}d\tau,
\end{equation}
then the following inequality holds
\begin{equation*}
y(t) \leq \bigg[(1+y(0)) \exp\bigg(\int_0^t [a(\tau)+b(\tau)]
e^{-\int_0^{\tau}a(s)ds}d\tau\bigg)\bigg]^{\exp(\int_0^t a(s)ds)}.
\end{equation*}
\end{lemma}

\begin{proof}
First note that for $y\geq 0$, it holds that $y (\ln y)_+<(1+y)\ln (1+y)$.
Then one has from (\ref{eq}) that
\begin{equation*}
\begin{aligned}
y(t)  %\leq y(0)+\int_0^t [a(\tau)(1+y \ln y)+b(\tau)y]d\tau\\
%& \leq y(0)+ \int_0^t [a(\tau)(1+y)+a(\tau)(1+y)\ln (1+y)+b(\tau)(1+y)]d\tau\\
& \leq y(0)+ \int_0^t [a(\tau)(1+y)\ln (1+y)+b(\tau)(1+y)]d\tau.
\end{aligned}
\end{equation*}
Let $X(t)=y(0)+ \int_0^t [a(\tau)(1+y)\ln (1+y)+b(\tau)(1+y)]d\tau$. Then we get $y(t)\leq X(t)$ and
\begin{eqnarray*}
\begin{aligned}
\frac{d}{dt}X(t)&=a(t)(1+y)\ln (1+y)+b(t)(1+y)\\
& \leq a(t)(1+X)\ln (1+X)+b(t)(1+X).
\end{aligned}
\end{eqnarray*}
Define $w(t)=1+X(t)$. The above inequality becomes
$
\frac{d w}{dt} \leq a(t)w \ln w+b(t)w
$
which gives by using the result of \cite[Lemma 4]{HKS1}
\begin{equation*}
w(t)\leq \bigg[w(0) \exp\bigg(\int_0^t b(\tau)
e^{-\int_0^{\tau}a(s)ds}d\tau\bigg)\bigg]^{\exp(\int_0^t a(s)ds)}.
\end{equation*}
Hence the facts $w(0)=1+X(0)=1+y(0)$ and $y(t)\leq X(t)\leq w(t)$ complete the proof.
\end{proof}

\noindent In the following, we show that for any $(t,x)\in[0,\infty) \times \R$, it holds that
$n=z\rho \in L^{\infty}([0,\infty); L^1(\R))$.
\begin{lemma}\label{estimatez1}
Let $f_0 \in L^1(\R\times V)$ and $z_0\in L^{\infty}(\R)$. Then the
solution $z$ of equation (\ref{eq2a8}) satisfies that
\begin{equation}\label{eq3a7}
\|n(t)\|_{L^1(\R)}=\|z\rho(t)\|_{L^1(\R)}\leq Ce^{-t}(1+t).
\end{equation}
\end{lemma}

{\it Proof}. Integrating equation (\ref{eq2a8}) with respect to $x$ gives
\begin{equation*}
\frac{d}{dt}\int_{\R}n(t,x)dx=\int_{\R}
\rho(t,x)g(S(t,x))dx-\int_{\R}n(t,x)dx.
\end{equation*}
Using the boundedness of $g$, we deduce that
\begin{equation}\label{eq3a8}
\frac{d}{dt}\int_{\R}n(t,x)dx+\int_{\R}n(t,x)dx\leq C
\|\rho\|_{L^1(\R)}.
\end{equation}
Integrating equation (\ref{eq2a7}) with respect to $v$ over $V$
gives the conservation of cell density
\begin{equation}\label{eq3a9}
\rho_t+ j_x=0,
\end{equation}
which immediately implies that
\begin{equation}\label{eq3a10}
\|\rho(t)\|_{L^1(\R)}=\|\rho_0\|_{L^1(\R)}=\|f_0\|_{L^1(\R\times V)}.
\end{equation}
Substituting (\ref{eq3a9}) into (\ref{eq3a8}) and applying the Gronwall's inequality, one deduces that
\begin{equation*}
\|n(t)\|_{L^1(\R)}=\|z\rho(t)\|_{L^1(\R)}\leq
Ce^{-t}(\|\rho_0\|_{L^1(\R)}t+\|n_0\|_{L^1(\R)}),
\end{equation*}
which implies (\ref{eq3a7}) due to the fact $n_0=z_0 \rho_0 \in L^1(\R)$ along with \eqref{eq3a10} and the condition in Lemma \ref{estimatez1}. \qed \\

With the equation (\ref{eq3a9}), we can write the equation (\ref{eq2a8}) as
\begin{equation}\label{eq3a10n}
\rho z_t+j z_x=\rho(g(S)-z).
\end{equation}
%If there is a point $(t_0, x_0)$ such that $\rho(t_0, x_0)=0$, then $f(t_0, x_0, v)=0$ for every $v \in V$ since $\rho(t_0,x_0)=\int_V f(t_0, x_0, v)dv$ denotes the total cell density at point $(t_0, x_0)$. Hence $j(t_0, x_0)=\int_V vf(t_0, x_0, v)dv=0$ and the equation (\ref{eq2a8}) is satisfied.
%Therefore in the sequel,  we assume that $\rho(t,x)\ne 0$ for all $(t,x)$ and rewrite (\ref{eq3a10n}) as
Recalling that $j(t,x)=\int_V vf(t,x,v)dv$ and $\rho(t,x)=\int_V
f(t,x,v)dv$, we see that $|\frac{j}{\rho}|<C(V)$ for some constant $C(V)$ depending on the measure of $V$. {\cred Because of the biological relevance, we are only interested in the case that $f, z$ and $S$ are all nonnegative. If there is a point $(t_0, x_0)$ such that $\rho(t_0, x_0)=0$, then $f(t_0, x_0, v)=0$ for every $v \in V$ since $\rho(t_0,x_0)=\int_V f(t_0, x_0, v)dv$ denotes the total cell density at point $(t_0, x_0)$. As a result, $j(t_0, x_0)=\int_V vf(t_0, x_0, v)dv=0$ and the equation (\ref{eq3a10n}) is satisfied in the solution space indicated in Theorem \ref{mth}. Therefore we only need to condider the case $\rho(t,x)>0$ and consequently} we can rewrite \eqref{eq3a10n} as
\begin{equation}\label{eq3a11}
z_t+\frac{j}{\rho} {  z_x}=g(S)-z.
\end{equation}
Then we have the following result for the solution of equation (\ref{eq3a11}).

\begin{lemma}\label{estimatez}
Let $z_0 \in L^{\infty}(\R)$. Then the solution $z$ of equation
(\ref{eq3a11}) satisfies
\begin{equation}\label{eq3a14}
\|z\|_{L^{\infty}(\R)}\leq C(1+e^{-t}).
\end{equation}
\end{lemma}
\begin{proof}
First we rewrite the equation (\ref{eq3a11}) as follows
\begin{equation}\label{eq3a11n}
\tilde{z}_t+\frac{j}{\rho} {   \tilde{z}_x}=e^tg(S)
\end{equation}
where $\tilde{z}=e^t z$. Then the characteristic equation of the hyperbolic equation
(\ref{eq3a11n}) is
\begin{equation*}
\frac{dx}{dt}=\frac{j(t,x)}{\rho(t,x)}=:\Lambda(t,x).
\end{equation*}
which satisfies $
0 \leq \Lambda(t,x) \leq C(V)
$ as argued above.
Hence along backward characteristics starting at $(t,x)$, we have
that for any $0\leq \tau\leq t$
\begin{equation}\label{eq3a13}
\mathbf{x}(\tau;t,x)=x-\int_{\tau}^t \Lambda(\tau,\textbf{x}(\tau))d\tau.
\end{equation}
%Recalling that $j(t,x)=\int_V vf(t,x,v)dv$ and $\rho(t,x)=\int_V
%f(t,x,v)dv$, we see that $\Lambda(t,x)$ is nonnegative and bounded by some constant
%$C(V)$ which depends on the size of domain $V$:
%$
%0 \leq \Lambda(t,x) \leq C(V).
%$
Then integrating the equation (\ref{eq3a11}) along the
characteristic curve (\ref{eq3a13}), one has that
\begin{equation*}
\tilde{z}(t,x)= z_0(\textbf{x}(0))+\int_0^t
e^{\tau} g(S(\tau,\textbf{x}(\tau)))d\tau.
\end{equation*}
 Noting that $g(S(t,x))$ is bounded for any $S(t,x)$, we have
\begin{equation*}
\|\tilde{z}(t,\cdot)\|_{L^{\infty}(\R)} \leq \|z_0\|_{L^{\infty}(\R)}+C \int_0^t
e^td\tau=\|z_0\|_{L^{\infty}(\R)}+C(e^t-1).
\end{equation*}
This indicates that
\begin{equation*}
\|{z}(t,\cdot)\|_{L^{\infty}(\R)} \leq e^{-t}(\|z_0\|_{L^{\infty}(\R)}+C(e^t-1))\leq  e^{-t}\|z_0\|_{L^{\infty}(\R)}+C
\end{equation*}
which completes the proof.
\end{proof}

The following Lemma gives {\it a priori} estimates for the $L^p$-norm $(1\leq p \leq \infty)$ of $S(t,x)$.

\begin{lemma}\label{1destimate1n}
Let $f_0 \in L^1(\R\times V), S_0\in L^p(\R) (1\leq p\leq \infty)$ and $z_0 \in L^{\infty}(\R)$.
Then there exists a constant $C$ such that the solution $S$ of equation
(\ref{eq2a9}) satisfies the following properties
\begin{equation}\label{eq4a1}
\|S(t)\|_{L^{p}(\R)} \leq C(1+\|\rho_0\|_{L^{1}(\R)}+\|S_0\|{  _{L^p(\R)}}), \ 1 \leq p \leq \infty.
\end{equation}
\end{lemma}

{\it Proof}.
%{  In the following proof, we let $ =1$ in the fundamental solution $\Gamma$ defined by (\ref{eq3a2}).}
% We consider a more general equation with positive $f,
%z$ and $S$,
%\begin{equation}\label{eq2a9}
%\frac{\partial S}{\partial t}=\Delta S+(g(S)-z) \rho-S(1+S).
%\end{equation}
Then by the Duhamel's principle, the solution of (\ref{eq2a9}) can be
implicitly represented as
\begin{eqnarray}\label{eq4a5}
\begin{array}{lll}
S(t,x)&=&\D\int_0^t
\Gamma(s,\cdot)\ast(g(S(t-s,\cdot))\rho(t-s,\cdot))ds-\int_0^t
\Gamma(s,\cdot)\ast (z\rho)(t-s,\cdot)ds\\
&& \D-\int_0^t \Gamma(s,\cdot)\ast S^2(t-s,\cdot)ds+\Gamma(t,\cdot)
\ast S_0(x)
%&=& S_1+S_2+S_3+S_4.
\end{array}
\end{eqnarray}
Notice that $\int_0^t \Gamma(s,\cdot)\ast S^2(t-s,\cdot)ds\geq 0$. Then by Jensen's inequality and convolution inequality, it follows from (\ref{eq4a5}) that
{  \begin{eqnarray}\label{eq4a6}
\begin{array}{lll}
\|S(t)\|_{L^{p}(\R)} &\leq \D \int_0^t
\|\Gamma(s,\cdot)\ast(g(S(t-s,\cdot))\rho(t-s,\cdot))\|_{L^{p}(\R)}ds\\
&+\D\int_0^t \|\Gamma(s,\cdot)\|_{L^p(\R)} \|z \rho(t-s,\cdot)\|_{L^1(\R)}ds+\|\Gamma(t,\cdot)\|_{L^1(\R)}
\|S_0\|_{L^p(\R)} \D\\
&\leq C \D \int_0^t \|\Gamma(s,\cdot)\|_{L^{p}(\R)}
(\|\rho(t-s,\cdot)\|_{L^{1}(\R)}+\|z\rho(t-s,\cdot)\|_{L^1(\R)})ds\\[3mm]
& \ \ \ +\|\Gamma(t,\cdot)\|_{L^1(\R)} \|S_0\|_{L^p(\R)},
\end{array}
\end{eqnarray}
where we have used the boundedness of $g$ assumed in (\ref{eq2a6}). It can be easily checked that $\max\limits_{0 \leq s\leq t}(1+t-s)e^{-(t-s)}=1$. Then for any $p$ with $p\in[1,\infty]$, we have from \eqref{eq3a10}, Lemma
\ref{estimatez1}  and Lemma \ref{fundsolu1} with $N=1$ that
\begin{eqnarray}\label{add}
\begin{aligned}
\D &\int_0^t \|\Gamma(s,\cdot)\|_{L^{p}(\R)}
(\|\rho(t-s,\cdot)\|_{L^{1}(\R)}+\|z\rho(t-s,\cdot)\|_{L^1(\R)})ds\\
&\leq C\int_0^t
\|\Gamma(s,\cdot)\|_{L^{p}(\R)}[1+(1+t-s)e^{-(t-s)}]ds\\
&\leq C\int_0^t
\|\Gamma(s,\cdot)\|_{L^{p}(\R)}dx
 \leq \infty.
\end{aligned}
\end{eqnarray}
}
Moreover, it is straightforward to
verify from \eqref{Gamma} that
\begin{equation}\label{eq4a7}
\|\Gamma(t,\cdot)\|_{L^1(\R)}=\frac{1}{(4\pi t)^{1/2}} e^{- t}
\|e^{-\frac{x^2}{4t}}\|_{L^1(\R)} \leq C e^{- t} \leq C.
\end{equation}
Then the inequality (\ref{eq4a1}) follows from (\ref{eq4a6}),
(\ref{add}) and (\ref{eq4a7}).

\qed

{\color{black} In the above proof of Lemma \ref{1destimate1n}, we only use the boundedness of $g$  in the assumption \eqref{eq2a6}. By the global boundedness of $S$ shown in Lemma \ref{1destimate1n},  $g'$ is bounded due to $g \in C^1$ as assumed in \eqref{eq2a6}. Therefore hereafter we shall use the boundedness of $g'$ directly.}

Below we derive some a {\it priori} $L^{\infty}$-estimates
on the gradient $S_x$ and $S_t$, which play key roles in the subsequent analysis.

\begin{lemma}\label{1destimate1}
Let $f_0 \in L^1(\R\times V)$, $S_0 \in W^{1,\infty}(\R)$ and $z_0 \in L^{\infty}(\R)$. Let $(f,z,S)$ satisfy \eqref{eq2a7}-\eqref{eq2a9}.
Then there exists a constant $C_0$ depending on $\|z_0\|_{L^{\infty}(\R)}, \|\rho_0\|_{L^1\cap L^{\infty}(\R)}$ such that the solution $S$ of
(\ref{eq2a9}) satisfies
%\begin{equation}\label{eq4a1}
%\|S(t)\|_{L^{p}(\R)} \leq C(1+\|\rho_0\|_{L^{1}(\R)}+\|S_0\|_{L^p(\R})),
%\ 1 \leq p \leq \infty,
%\end{equation}
\begin{equation}\label{eq4a2}
\begin{array}{lll}
\D\|S_x(t)\|_{L^{\infty}(\R)} \leq C_0(1+(\ln t)_++(\ln \sup\limits_{0\leq s\leq t} \|\rho(s)\|_{L^2(\R)})_+).
\end{array}
\end{equation}
\end{lemma}

{\it Proof}. First note that the solution $S(t,x)$ can be expressed by (\ref{eq4a5}). To derive the $L^{\infty}$ bound for $S_x(t,x)$, we first estimate $\|\xi\hat{S}(t,\xi)\|_{L^1(\R)}$ whereby \
$\hat{}$ \  denotes the Fourier transform with frequency $\xi$, and then use the Fourier transform inequality
$\|S_x(t,\cdot)\|_{L^{\infty}(\R)}\leq \|\hat{S}_x(t,\cdot)\|_{L^1(\R)} = \|\xi\hat{S}(t,\xi)\|_{L^1(\R)}$ to
obtain (\ref{eq4a2}).  To this end, we take the Fourier transform on both side of (\ref{eq4a5}) and get
\begin{eqnarray}\label{eq4a8}
\begin{array}{lll}
\hat{S}(t,\xi)&=&\D\int_0^t \hat{\Gamma}(s,\xi)\widehat{(g
\rho)}(t-s,\xi)ds-\int_0^t
\hat{\Gamma}(s,\xi)\widehat{(z\rho)}(t-s,\xi)ds\\
&&\D-\int_0^t
\hat{\Gamma}(s,\xi)\widehat{S^2}(t-s,\xi)ds+\widehat{\Gamma\ast{S}_0}\\
&=& \hat{S}_1+\hat{S}_2+\hat{S}_3+\hat{S}_4.
\end{array}
\end{eqnarray}
Noticing that  $\|\xi\hat{S}(t,\xi)\|_{L^1(\R)} =\|\xi(\hat{S}_1+\hat{S}_2+\hat{S}_3+\hat{S}_4)(t,\xi)\|_{L^1(\R)}\leq \sum_{i=1}^4\|\xi \hat{S}_i(t,\xi)\|_{L^1(\R)}$,
we next estimate $\|\xi \hat{S}_i(t,\xi)\|_{L^1(\R)}$ for $i=1,2,3,4$.
We first estimate $\|\xi \hat{S}_1(t,\xi)\|_{L^1(\R)}$ for which we apply the idea of \cite{HKS3}
and split the integral into two parts as follows
\begin{eqnarray}\label{eq4a9}
\|\xi \hat{S}_1(t,\cdot)\|_{L^1(\R)} &\leq& \D  \int_0^t
\|\xi\hat{\Gamma}(s,\xi)\widehat{(g
\rho)}(t-s,\xi)\|_{L^1(\R)}ds=\int_0^r \cdots +\int_r^t \cdots,
\end{eqnarray}
where $r$ is between $0$ and $t$ and will be appropriately determined later. Note that the Fourier transform of
green function $\Gamma$ is $\hat{\Gamma}(s,\xi)=\exp(-s(4 \xi^2+1))$. Then for $0
<s <r$, we apply H\"{o}lder inequality, Plancherel's identity and
boundedness of $g$ to deduce that
\begin{eqnarray}\label{eq4a10}
\begin{array}{lll}
&\D\int_0^r \|\xi\hat{\Gamma}(s,\xi)\widehat{(g
\rho)}(t-s,\xi)\|_{L^1(\R)}ds\\
&\leq \D  \int_0^r \|\xi\hat{\Gamma}(s,\xi)\|_{L^2(\R)}
\|\widehat{g
\rho}(t-s,\xi)\|_{L^2(\R)}ds\\
&\leq \D C \|g\|_{L^{\infty}(\R)} \sup\limits_{0\leq s\leq t}
\|\rho(s, \cdot)\|_{L^2(\R)} \int_0^r e^{-s}s^{-3/4}ds\\
&\leq \D C r^{1/4}\sup\limits_{0\leq s\leq t}
\|\rho(s,\cdot)\|_{L^2(\R)},
\end{array}
\end{eqnarray}
where we have used the inequality
\begin{eqnarray*}\label{eq4a16}
\begin{array}{lll}
\|\xi\hat{\Gamma}(s,\xi)\|_{L^2(\R)}=\|\xi\exp(-s(4 \xi^2+1))\|_{L^2(\R)}
\leq C e^{-s} s^{-3/4},
\end{array}
\end{eqnarray*}
which can be derived directly by the simple calculations. For
$0<r<s<t$, we use the mass conservation of $\rho$ from \eqref{eq3a9} and boundedness of
$g$ from \eqref{eq2a6} to infer that
\begin{eqnarray}\label{eq4a11}
\begin{array}{lll}
&\D\int_r^t\|\xi\hat{\Gamma}(s,\xi)\widehat{(g
\rho)}(t-s,\xi)\|_{L^1(\R)}ds \\
& \D \leq  \int_r^t
\|\xi\hat{\Gamma}(s,\xi)\|_{L^1(\R)} \|\widehat{g
\rho}(t-s,\xi)\|_{L^{\infty}(\R)}ds\\
&\leq \D C \|\rho_0\|_{L^1(\R)} \int_r^t e^{-s}s^{-1}ds\\
&\leq \D C \|\rho_0\|_{L^1(\R)}(\ln t-\ln r)
\end{array}
\end{eqnarray}
where we have used the fact that $\|\widehat{g\rho}\|_{L^{\infty}(\R)}\leq
\|g\rho\|_{L^1(\R)}\leq \|g\|_{L^{\infty}(\R)} \|\rho\|_{L^1(\R)} \leq C \|\rho\|_{L^1(\R)}$.
Then the combination of (\ref{eq4a10}) and (\ref{eq4a11}) leads  to
\begin{eqnarray}\label{eq4a12}
\|\xi \hat{S}_1(t,\cdot)\|_{L^1(\R)} &\leq& \D
C(r^{1/4}\sup\limits_{0\leq s\leq t}
\|\rho(s,\cdot)\|_{L^2(\R)}+\|\rho_0\|_{L^1(\R)}(\ln t-\ln r)).
\end{eqnarray}
Now choosing $r=\min\{t, (\sup\limits_{0\leq s\leq
t}\|\rho(s,\cdot)\|_{L^2(\R)})^{-4}\}$ in (\ref{eq4a12}), we
have
\begin{eqnarray}\label{eq4a13}
\|\xi \hat{S}_1(t,\cdot)\|_{L^1(\R)} &\leq& \D
C(1+\|\rho_0\|_{L^1(\R)}[1+(\ln t)_++(\ln \sup\limits_{0\leq s\leq t} \|\rho(s)\|_{L^2(\R)})_+])
\end{eqnarray}
Thus the $L^1$ estimate for $\xi \hat{S}_1(t,\xi)$ is completed.

Performing the same procedure as above (replacing $g$ by $z$), we have the $L^1$ bound for $\xi\hat{S}_2(t,\xi)$ as follows
\begin{eqnarray*}\label{eq4a14}
\begin{array}{lll}
\|\xi \hat{S}_2(t,\cdot)\|_{L^1(\R)} &\leq& \D C(1+(\ln
t)_++(\ln\sup\limits_{0\leq s\leq t} \|\rho(s)\|_{L^2(\R)})_+),
\end{array}
\end{eqnarray*}
where Lemma \ref{estimatez1} and Lemma \ref{estimatez} have been used.

Moreover, using the Plancherel's identity and inequality
(\ref{eq4a1}), we have
\begin{eqnarray}\label{eq4a15}
\begin{array}{lll}
\|\xi \hat{S}_3(t,\cdot)\|_{L^1(\R)} &\leq& \D \int_0^t
\|\xi\hat{\Gamma}(s,\xi)\|_{L^2(\R)}\|S^2(t-s,\xi)\|_{L^2(\R)}ds\\
&\leq & \D \int_0^t
\|\xi\hat{\Gamma}(s,\xi)\|_{L^2(\R)}\|S(t-s,\xi)\|_{L^4(\R)}^2ds\\[3mm]
&\leq & C(1+\|S_0\|_{L^4(\R)}+\|\rho_0\|_{L^{1}(\R)})^2
\end{array}
\end{eqnarray}
where we have used the fact $\int_0^t
\|\xi\hat{\Gamma}(s,\xi)\|_{L^2(\R)}\leq C\int_0^t e^{-s}s^{-3/4}ds<\infty$.
In addition, {  from \eqref{Gamma}, we have  $\|\Gamma(t,\cdot)\|_{L^1(\R)} \leq C$ for some constant $C>0$.} Then with the convolution inequality, the following holds
\begin{equation}\label{eq4a17}
\|\partial_x S_4\|_{L^{\infty}(\R)}=\|\Gamma \ast \partial_x
S_0\|_{L^{\infty}(\R)} \leq \|\Gamma(t,\cdot)\|_{L^1(\R)} \cdot \|
\partial_x S_0\|_{L^{\infty}(\R)}
\leq C \|\partial_x S_0\|_{L^{\infty}(\R)}.
\end{equation}
Finally, we derive the $L^{\infty}$ norm of $S_x$ as follows
\begin{eqnarray}\label{eq4a18}
\begin{array}{lll}
\D\|S_x\|_{L^{\infty}(\R)}
&\leq& \D\sum\limits_{i=1}^{3} \|\partial_x S_i(t)\|_{L^{\infty}(\R)}+\|\partial_x S_4(t)\|_{L^{\infty}(\R)}\\
&\leq& \D\sum\limits_{i=1}^{3} \|\widehat{\partial_x S_i}\|_{L^{1}(\R)}+\|\partial_x S_4(t)\|_{L^{\infty}(\R)}\\
&\leq& \D\sum\limits_{i=1}^{3} \|\xi
\hat{S}_i(t,\xi)\|_{L^1(\R)}+\|\partial_x S_4(t)\|_{L^{\infty}(\R)}.
\end{array}
\end{eqnarray}
Then the substitution of (\ref{eq4a13})-(\ref{eq4a17}) into (\ref{eq4a18}) gives inequality
(\ref{eq4a2}) and hence completes the proof of Lemma 4.1.

\qed

The next Lemma gives the {\it a priori} $L^{\infty}$-estimate  for the temporal gradient $S_t$.
\begin{lemma}\label{1destimate2}
Let $f_0 \in L^1\cap L^\infty(\R\times V)$, $S_0 \in W^{2,\infty}(\R)$ and $z_0 \in L^{\infty}(\R)$. Let $(f,z,S)$ satisfy \eqref{eq2a7}-\eqref{eq2a9}. Then the
solution $S$ of (\ref{eq2a9}) satisfies that
\begin{equation}\label{eq4a3}
\begin{array}{lll}
\D\|S_t(t)\|_{L^{\infty}(\R)} \leq C (1+(\ln t)_++(\ln \sup\limits_{0\leq s\leq t} \|\rho(s)\|_{L^2(\R)})_+).
\end{array}
\end{equation}
\end{lemma}

{\it Proof}. The idea used in the proof of Lemma \ref{1destimate1} will be
partially applied here. First we set $\tilde{v}=S_t$ and differentiate equation (\ref{eq2a9}) with respect to $t$ to obtain an equation for
$\tilde{v}$
\begin{equation*}
\frac{\partial \tilde{v}}{\partial t}=\Delta \tilde{v}-\tilde{v}+\frac{\partial}{\partial t} \mathcal{F}(S,z, \rho)-2S\tilde{v},
\end{equation*}
where $\mathcal{F}(S,z, \rho)=(g(S)-z)\rho$. Then by the Duhamel's principle we have
\begin{eqnarray}\label{expressiontildev}
\begin{aligned}
\tilde{v}(t,x)&=\Gamma(t,\cdot) \ast  \tilde{v}_0(\cdot)-2\int_0^t \Gamma(s,\cdot)\ast(S\tilde{v})(t-s)ds
+\int_0^t \Gamma(s,\cdot) \ast \mathcal{F}_t(t-s,\cdot)ds\\
&=I_1+I_2+I_3.
\end{aligned}
\end{eqnarray}
Next we  employ a similar argument as in the proof of
Lemma \ref{1destimate1} for $S_x$ to estimate $\tilde{v}$. First from the equation (\ref{eq2a9}), {  one} has that $S_t(0)=\tilde{v}_0=S_{0xx}+(g(S_0)-z_0) \rho_0-S_0-S_0^2$.
Then $\|\tilde{v}_0\|_{L^{\infty}(\R)} \leq
C(1+\|S_0\|_{W^{2,\infty}(\R)}
+\|\rho_0\|_{L^{\infty}(\R)}+\|z_0\rho_0\|_{L^{\infty}(\R)})$ by the boundedness of $g$. Therefore  the
convolution inequality yields the following estimate
\begin{eqnarray}\label{estimateI1}
\begin{aligned}
\|I_1\|_{{  L^\infty(\R)}} &\leq \|\Gamma(t)\|_{L^1(\R)} \|\tilde{v}_0\|_{L^{\infty}(\R)}\\
&\leq Ce^{-t}(1+\|S_0\|_{W^{2,\infty}(\R)}+\|\rho_0\|_{L^{\infty}(\R)}+\|z_0\rho_0\|_{L^{\infty}(\R)}),
\end{aligned}
\end{eqnarray}
where we have used the fact $\|\Gamma(t)\|_{L^1(\R)}=e^{-t}$ which is obtained in Lemma \ref{fundsolu1}.

For $I_2$, we use the convolution inequality again and (\ref{eq4a1}) to get
\begin{eqnarray}\label{estimateI2}
\begin{aligned}
\|I_2\|_{{  L^\infty(\R)}}  \leq \int_0^t \|\Gamma(s)\|_{L^1(\R)}\|S\|_{L^{\infty}(\R)} \|\tilde{v}\|_{L^{\infty}(\R)}ds
 \leq C \int_0^t e^{-s} \|\tilde{v}\|_{L^{\infty}(\R)}ds,
\end{aligned}
\end{eqnarray}
where $C$ depends on $\|S_0\|_{L^{\infty}(\R)},\|\rho_0\|_{L^{1}(\R)}$ which can be seen from (\ref{eq4a1}) directly.

Next we prove the boundedness of $I_3$ using the Fourier transform inequality $\|I_3\|_{L^{\infty}(\R)}
\leq \|\hat{I_3}\|_{L^1(\R)}$. To this end, we write $I_3$ in the form
\begin{eqnarray}\label{expressionI3}
\begin{aligned}
I_3&= \int_0^t \Gamma(s,\cdot)\ast (g(S) \rho_t(t-s,\cdot))ds-\int_0^t
\Gamma(s,\cdot)\ast(z \rho)_t(t-s,\cdot)ds \\
& \ \ +\int_0^t \Gamma(s,\cdot)\ast (g'(S)\rho \tilde{v}(t-s,\cdot))ds\\
&=M_1+M_2+M_3.
\end{aligned}
\end{eqnarray}
Since $\|I_3\|_{L^{\infty(\R)} }\leq
\sum_{i=1}^3\|M_i\|_{L^{\infty}(\R)} \leq
\sum_{i=1}^3\|\widehat{M_i}\|_{L^1(\R)}$, it suffices to estimate
$\|\widehat{M_i}\|_{L^1(\R)}$ for $i=1,2,3$. From the equation (\ref{eq2a7}), we have $\D \rho_t=-j_x$,
where $j(t,x)=\int_V v f(t,x,v)dv$ denoting the density flux. Then it
follows that
\begin{eqnarray}\label{eq4a21}
\begin{array}{lll}
\widehat{g \frac{\partial \rho}{\partial t}}(t-s,\xi)&=&-\widehat{g
\frac{\partial j}{\partial x}}(t-s,\xi)=-\hat{g}(t-s,\cdot) \ast
\widehat{\frac{\partial j}{\partial x}}(t-s,\cdot)\\
&=& -\hat{g}(t-s,\cdot) \ast (i \xi \hat{j}(t-s,\cdot))\\
&=& -i\xi\int_{\R} \hat{g}(t-s,\xi-y)\hat{j}(t-s,y)dy\\
&=& -i\xi (\hat{g}\ast \hat{j})(t-s,\xi)\\
&=&-i\xi \widehat{gj}(t-s,\xi).
\end{array}
\end{eqnarray}
In addition, from the Plancherel's identity, using the positivity of $f$, we can deduce that
\begin{eqnarray}\label{eq4a22}
\begin{array}{lll}
\|\widehat{j}(t)\|_{L^2(\R)} =\D\|{j}(t)\|_{L^2(\R)} \leq \bigg(
\int_{\R} \bigg(\int_{V}v
f(t,x,v)dv\bigg)^2dx\bigg)^{1/2}{  \leq} C(V) \|\rho(t)\|_{L^2(\R)}
\end{array}
\end{eqnarray}
where we have used the compactness of domain $V$.

Observing that
\begin{eqnarray}\label{eq4a23}
\begin{array}{lll}
\|\widehat{M}_1\|_{L^1(\R)}\leq  \D \int_0^t \|\xi
\widehat{\Gamma}(s,\xi) \widehat{(gj)}(t-s,\xi)\|_{L^1(\R)}ds,
\end{array}
\end{eqnarray}
we use exactly the same method as estimating (\ref{eq4a9}), and employ
H\"{o}lder inequality, Plancherel's inequality and the boundedness
of $g$ to (\ref{eq4a23}). After some calculations, we have
\begin{eqnarray}\label{estimateM1}
\begin{array}{lll}
\|M_1\|_{L^{\infty}(\R)}\leq \|\widehat{M}_1\|_{L^1(\R)} \leq
C(1+\|\rho_0\|_{L^1(\R)}(1+(\ln t)_++(\ln\sup\limits_{0\leq s\leq t}
\|\rho(s)\|_{L^2(\R)})_+)).
\end{array}
\end{eqnarray}
To estimate $M_2$, we first notice from the equation (\ref{eq2a8}) that
\begin{eqnarray*}\label{eq4a25}
\begin{array}{lll}
\D \widehat{\frac{\partial (z\rho)}{\partial t}}=\D -\widehat{\frac{\partial
(zj)}{\partial x}}+ \widehat{\rho g(S)}-\widehat{z\rho}
= i \xi \widehat{z j}(t,\xi)+\widehat{\rho
g(S)}(t,\xi)-\widehat{z\rho}(t,\xi).
\end{array}
\end{eqnarray*}
Then it follows  from the convolution identity involving the Fourier transform that
\begin{eqnarray}\label{eq4a26}
\begin{array}{lll}
\|\widehat{M}_2\|_{L^1(\R)} &\leq& \D \int_0^t \|\xi
\widehat{\Gamma}(s,\xi) \widehat{zj}(t-s,\xi)\|_{L^1(\R)}ds+\D
\int_0^t \|\widehat{\Gamma}(s,\xi) \widehat{\rho
g(S)}(t-s,\xi)\|_{L^1(\R)}ds \\
&& +\D \int_0^t \|\widehat{\Gamma}(s,\xi)
\widehat{z\rho}(t-s,\xi)\|_{L^1(\R)}ds.
\end{array}
\end{eqnarray}
Note that one has $\|j\|_{L^1(\R)} \leq C\|\rho\|_{L^1(\R)}\leq C \|\rho_0\|_{L^1(\R)}$ from \eqref{flux} and \eqref{eq3a9}. Furthermore
$\|\widehat{zj}\|_{L^{\infty}(\R)} \leq \|zj\|_{L^1(\R)}$. Then
applying inequality (\ref{eq3a7}), (\ref{eq3a14}) and
(\ref{eq4a22}), and using the same approach as estimating (\ref{eq4a9}),
we end up with the following inequality
\begin{eqnarray}\label{eq4a27}
\begin{aligned}
\D\int_0^t \|\xi \widehat{\Gamma}(s,\xi)&
\widehat{zj}(t-s,\xi)\|_{L^1(\R)}ds \\
\leq & C(1+(\ln t)_++(\ln \sup\limits_{0\leq s\leq t} \|\rho(s)\|_{L^2(\R)})_+).
\end{aligned}
\end{eqnarray}
Noticing that $\|\widehat{\Gamma}(t,\cdot)\|_{L^1(\R)} \leq C
t^{-1/2}e^{-t}$, one has
\begin{eqnarray}\label{eq4a28}
\begin{array}{lll}
\D \int_0^t \|\widehat{\Gamma}(s,\xi) \widehat{\rho
g(S)}(t-s,\xi)\|_{L^1(\R)}ds&\leq& \D \sup\limits_{0\leq s \leq
t}\|\widehat{g\rho}(t-s,\cdot)\|_{L^{\infty}(\R)} \int_0^t
\|\widehat{\Gamma}(\tau,\cdot)\|_{L^1(\R)}d\tau\\
&\leq& C \sup\limits_{0\leq s \leq
t}\|g\rho(t-s,\cdot)\|_{L^1(\R)}\\
&\leq& C \|\rho_0\|_{L^1(\R)}.
\end{array}
\end{eqnarray}
Similarly, we can deduce from (\ref{eq3a14}) that
\begin{eqnarray}\label{eq4a29}
\begin{array}{lll}
\D \int_0^t \|\widehat{\Gamma}(s,\xi) \widehat{z\rho
}(t-s,\xi)\|_{L^1(\R)}ds &\leq& C \|\rho_0\|_{L^1(\R)}.
\end{array}
\end{eqnarray}
Then the substitution of (\ref{eq4a27}), (\ref{eq4a28}) and
(\ref{eq4a29}) into (\ref{eq4a26}) yields
\begin{eqnarray}\label{estimateM2}
\begin{array}{lll}
\|M_2\|_{L^{\infty}(\R)} \leq \|\widehat{M}_2\|_{L^1(\R)}
 \leq
C(1+(\ln t)_++(\ln \sup\limits_{0\leq s\leq t} \|\rho(s)\|_{L^2(\R)})_+).
\end{array}
\end{eqnarray}
To finish the proof, it remains to estimate $M_3$ for which we have
\begin{eqnarray}\label{estimateM3}
\begin{array}{lll}
\|M_3\|_{L^{\infty}(\R)} &\leq& \D \int_0^t \|g'(S) \rho\|_{L^1(\R)}
\|\Gamma(s,\cdot)\|_{L^{\infty}(\R)} \|\tilde{v}\|_{L^{\infty(\R)}}ds\\
&\leq& C \D \|\rho_0\|_{L^1(\R)} \int_0^t
t^{-1/2}e^{-s}\|\tilde{v}\|_{L^{\infty}(\R)}ds.
\end{array}
\end{eqnarray}
Feeding (\ref{estimateM1}), (\ref{estimateM2}) and (\ref{estimateM3}) into (\ref{expressionI3}) and
combining the resulting inequality with (\ref{estimateI1}) and (\ref{estimateI2}),
we end up with the following inequality from (\ref{expressiontildev})
\begin{eqnarray*}\label{eq4a33}
\begin{aligned}
\D\|\tilde{v}\|_{L^{\infty}(\R)}
\leq C(1+(\ln t)_++(\ln \sup\limits_{0\leq s\leq t} \|\rho(s)\|_{L^2(\R)})_+)
+C \int_0^t
(e^{-s}+s^{-1/2}e^{-s})\|\tilde{v}\|_{L^{\infty}(\R)}ds,
\end{aligned}
\end{eqnarray*}
where $C>0$ is a constant depending on initial data.

Note that $\int_0^t(e^{-s}+s^{-1/2}e^{-s})ds \leq \int_0^{\infty} (e^{-s}+s^{-1/2}e^{-s})ds=1+\sqrt{\pi}/2$ and $\tilde{v}=S_t$. Then
the application of Gronwall's inequality into the above inequality gives (\ref{eq4a3}) due to $S_t=\tilde{v}$.
The proof of Lemma \ref{1destimate2} is completed.
\qed

\section{Global existence}
\setcounter{equation}{0}
\renewcommand{\theequation}{\thesection.\arabic{equation}}
The proof of global existence of solutions to (\ref{eq2a7})-(\ref{eq2a9}) consists of {  a local existence theorem} and the {\it a priori} estimates. We first prove the local existence of solutions.
\subsection{Local existence}

%Due to the priori estimates given in Lemma \ref{estimatez} for $z$, it remains to consider equations (\ref{eq2a7}) and (\ref{eq2a9}).
The local existence theorem is given below.
\begin{lemma}[Local existence]\label{local}
Let $f_0\in L^1\cap L^{\infty}(\R \times V), z_0 \in L^{\infty}(\R), S_0 \in W^{2,\infty}(\R)$. Let the hypothesis (H) hold. Then there exists a positive constant $T_0$ such that (\ref{eq2a7})-(\ref{eq2a9}) has a unique solution satisfying $f \in L^\infty([0, T_0); L^1\cap L^{\infty}(\R \times V))$, $S \in L^\infty([0, T_0); W^{2,\infty}(\R))$ with $z \in L^{\infty}([0,T_0);L^{\infty}(\R))$.
\end{lemma}

\begin{proof}
 From the results given in Lemma \ref{estimatez}, the local solution $z$ can be obtained by the fixed point theorem directly irrespective of the properties of $f$ and $S$ due to the uniform boundedness of $g$ and $\frac{j}{\rho}$.
 Therefore we only consider the equations for $f$ and $S$. The proof consists of the following three steps. For convenience, we denote for some $T>0$
\begin{equation*}
\mathcal{X}(T)=L^\infty([0, T); L^{\infty}(\R \times V)),\ \ \|f\|_{\mathcal{X}}=\sup \limits_{0 \leq t\leq T} \|f(t, \cdot, \cdot)\|_{L^{\infty}(\R \times V)}
\end{equation*}
and
\begin{equation*}
\mathcal{Y}(T)=L^\infty([0, T); W^{1,\infty}(\R)), \ \ \|S\|_{\mathcal{Y}}=\sup \limits_{0 \leq t\leq T} \| S(t, \cdot)\|_{W{^{1,\infty}(\R)}}.
\end{equation*}

{\it Step 1}. Given a function $\theta\in \mathcal{X}(T)$ with $\varrho(t,x)=\int_V \theta(t,x,v)dv$,  we consider the following equation
\begin{eqnarray}\label{leqn5}
\begin{aligned}
S_t=S_{xx}-S+g(S)-S^2-z\varrho, \  \ S|_{t=0}=S_0.
\end{aligned}
\end{eqnarray}
For convenience, we denote $h(S)=g(S)-S^2$. Then by Duhamel principle, one can write the solution of (\ref{leqn5}) as
\begin{equation}\label{semi}
S(t,x)=\Gamma(t,\cdot)\ast S_0(\cdot)+\int_0^t \Gamma(\tau, \cdot)\ast [h(S(t-\tau, \cdot))-(z \varrho)(t-\tau, \cdot)]d\tau.
\end{equation}
Next we define an operator $\mathcal{T}_1: \mathcal{Y} \to \mathcal{Y}$ such that
\begin{equation*}\label{T1}
\mathcal{T}_1[S]=\Gamma(t,\cdot)\ast S_0(\cdot)+\int_0^t \Gamma(\tau, \cdot)\ast [h(S(t-\tau, \cdot))-(z \varrho)(t-\tau, \cdot)]d\tau.
\end{equation*}
and show that $\mathcal{T}_1$ has a unique fixed point. Indeed for any $S, \tilde{S} \in \mathcal{Y}$ with $S(0,\cdot)=\tilde{S}(0, \cdot)=S_0$,  we have
\begin{equation}\label{leqn5nn*}
\mathcal{T}_1[S]-\mathcal{T}_1[\tilde{S}]=\int_0^t \Gamma(\tau, \cdot)\ast[h(S(t-\tau, \cdot))-h(\tilde{S}(t-\tau, \cdot))]d\tau.
\end{equation}
By the convolution inequality, we have the following estimates
\begin{eqnarray}\label{leqn6}
\begin{aligned}
&\|\Gamma(\tau, \cdot) \ast \{h(S(t-\tau, \cdot))-h(\tilde{S}(t-\tau, \cdot))\}\|_{\mathcal{Y}}\\
&=\sup\limits_{0\leq \tau \leq t}\|(\Gamma(\tau, \cdot)+\partial_x \Gamma(\tau, \cdot))\ast [h(S(t-\tau, \cdot))-h(\tilde{S}(t-\tau, \cdot))]\|_{L^{\infty}(\R)}\\
& \leq \sup\limits_{0\leq \tau \leq t} \Big(\|\Gamma(\tau, \cdot)\|_{W^{1,1}(\R)}\|h(S(t-\tau, \cdot))-h(\tilde{S}(t-\tau, \cdot)\|_{L^{\infty}(\R)}\Big).
\end{aligned}
\end{eqnarray}
Note that
\begin{eqnarray}
\begin{aligned}
%h(S, \varrho)-h(\tilde{S},\tilde{\varrho})&=(g(S)-z)\varrho-S^2-(g(\tilde{S})-z)\tilde{\varrho}+\tilde{S}^2\\
%%&=(g(S)-z)(\varrho-\tilde{\varrho})+[(g(S)-z)-(g(\tilde{S})-z)]\tilde{\varrho}+(S+\tilde{S})(S-\tilde{S})\\
%&=(g(S)-z)(\varrho-\tilde{\varrho})+[g(S)-g(\tilde{S})]\tilde{\varrho}+(S+\tilde{S})(S-\tilde{S})\\
%&=(g(S)-z)(\varrho-\tilde{\varrho})+g'(\xi)(S-\tilde{S})\tilde{\varrho}+(S+\tilde{S})(S-\tilde{S})\\
%&=(g(S)-z)(\varrho-\tilde{\varrho})+(g'(\xi)\tilde{\varrho}+S+\tilde{S})(S-\tilde{S})
h(S)-h(\tilde{S})=g(S)-S^2-(g(\tilde{S})-\tilde{S}^2)=[g'(\xi)-(S+\tilde{S})](S-\tilde{S})
\end{aligned}
\end{eqnarray}
where $\xi$ is between $S$ and $\tilde{S}$. Therefore it follows from Lemma \ref{1destimate1n} and (\ref{eq2a6}) that
\begin{eqnarray}\label{leqn7}
\begin{aligned}
&\|h(S)-h(\tilde{S})\|_{L^{\infty}(\R)} \leq C\|S-\tilde{S}\|_{L^{\infty}(\R)}.
\end{aligned}
\end{eqnarray}
From \eqref{Gamma}, one has that {  $\|\Gamma(\tau, \cdot)\|_{W^{1,1}(\R)} \leq Ce^{-t}(1+t^{-\frac{1}{2}})$ and hence $\int_0^t\|\Gamma(\tau, \cdot)\|_{W^{1,1}(\R)}d\tau\leq C(1-e^{-t}+t^{\frac{1}{2}})$}. Then it follows from (\ref{leqn5nn*})-(\ref{leqn7}) that
\begin{eqnarray*}\label{leqn8}
\begin{aligned}
&\|\mathcal{T}_1[S]-\mathcal{T}_1[\tilde{S}]\|_{\mathcal{Y}}
&\leq C m(t)\|S-\tilde{S}\|_{\mathcal{Y}}
\end{aligned}
\end{eqnarray*}
where $m(t)=1-e^{-t}+t^{\frac{1}{2}}$. By choosing $T$ small with $0<t<T$ such that $C m(t)<1$, we conclude that $\mathcal{T}_1$ is a contraction mapping on $\mathcal{Y}(T)$ and hence has a unique fixed point $S \in \mathcal{Y}(T)$. This implies that the problem \eqref{leqn5} has a unique solution $S\in \mathcal{Y}(T)$ if $T$ is small, which further satisfies the following estimates in view of Lemma \ref{1destimate1} and Lemma \ref{1destimate2}
\begin{eqnarray}\label{pes}
\begin{aligned}
\|S_{x}(t)\|_{L^{\infty}(\mathbb{R})}+\|S_{t}(t)\|_{L^{\infty}(\mathbb{R})} &\leq C(1+(\ln t)_{+}+(\ln \sup _{0 \leq s \leq t}\|\varrho(s)\|_{L^{2}(\mathbb{R})})_+)\\
%&\leq C(1+\ln(1+t)+\ln (1+\sup _{0 \leq s \leq t}\|\rho(s)\|_{L^{2}(\mathbb{R})}))
&\leq C(1+t+\sup _{0 \leq s \leq t}\|\varrho(s)\|_{L^2(\mathbb{R})}) \leq C(1+t)
\end{aligned}
\end{eqnarray}
where we have used the fact $(\ln s)_+\leq \ln(1+s)\leq s$ for $s\geq 0$ and $\|\varrho\|_{L^2(\R)} \leq C(V) \|\theta\|_{L^2(\R\times V)}$ by the H\"{o}lder inequality due to the definition $\varrho(t,x)=\int_V \theta(t,x,v)dv$ for $\theta \in \mathcal{X}(T)$.

Next we proceed to explore the continuous dependence of $S$ on $\theta$. For this purpose, we choose another $\bar{\theta}(t,x,v) \in \mathcal{X}(T)$ with $\tilde{\varrho}(t,x)=\int_V \bar{\theta}(t,x,v)dv$, and by $\tilde{S}$ we denote the solution of \eqref{leqn5} with $\varrho$ replaced by $\bar{\varrho}$. Then it follows from \eqref{leqn5} and \eqref{semi} that
\begin{equation}\label{leqn5n*}
S-\bar{S}=\int_0^t \Gamma(\tau, \cdot)\ast\{\hbar(S(t-\tau, \cdot), \varrho(t-\tau, \cdot))-\hbar(\bar{S}(t-\tau, \cdot), \bar{\varrho}(t-\tau, \cdot))\}d\tau
\end{equation}
where $\hbar(s,\varrho)=g(S)-S^2-z\varrho$.
The convolution inequality yields that
\begin{eqnarray}\label{leqn6*}
\begin{aligned}
&\|\Gamma(\tau, \cdot) \ast (\hbar(S, \varrho)-\hbar(\bar{S},\bar{\varrho}))\|_{\mathcal{Y}}\\
&=\sup\limits_{0\leq \tau \leq t}\|(\Gamma(\tau, \cdot)+\Gamma_x(\tau, \cdot))\ast (\hbar(S, \varrho)-\hbar(\bar{S},\bar{\varrho}))\|_{L^{\infty}(\R)}\\
& \leq \sup\limits_{0\leq \tau \leq t} \Big(\|\Gamma(\tau, \cdot)\|_{W^{1,1}(\R)}\|(\hbar(S, \varrho)-\hbar(\bar{S},\bar{\varrho}))\|_{L^{\infty}(\R)}\Big).
\end{aligned}
\end{eqnarray}
Note that
\begin{eqnarray}
\begin{aligned}
\hbar(S, \varrho)-\hbar(\bar{S},\bar{\varrho})&=(g(S)-z)\rho-S^2-(g(\bar{S})-z)\bar{\varrho}+\bar{S}^2\\
%&=(g(S)-z)(\rho-\bar{\varrho})+[(g(S)-z)-(g(\bar{S})-z)]\bar{\varrho}+(S+\bar{S})(S-\bar{S})\\
&=(g(S)-z)(\varrho-\bar{\varrho})+[g(S)-g(\bar{S})]\bar{\varrho}-(S+\bar{S})(S-\bar{S})\\
&=(g(S)-z)(\varrho-\bar{\varrho})+g'(\zeta)(S-\bar{S})\bar{\varrho}-(S+\bar{S})(S-\bar{S})\\
&=(g(S)-z)(\varrho-\bar{\varrho})-(S+\bar{S}-g'(\zeta)\bar{\varrho})(S-\bar{S})
\end{aligned}
\end{eqnarray}
where $\zeta$ is between $S$ and $\bar{S}$. Therefore it follows from Lemma \ref{estimatez}, Lemma \ref{1destimate1n} and (\ref{eq2a6}) that
\begin{eqnarray}\label{leqn7*}
\begin{aligned}
\|(\hbar(S, \varrho)-\hbar(\bar{S},\bar{\varrho}))\|_{L^{\infty}(\R)}
\leq C(1+e^{-t})(\|\varrho-\bar{\varrho}\|_{L^{\infty}(\R)}+\|S-\bar{S}\|_{L^{\infty}(\R)}).
\end{aligned}
\end{eqnarray}
Note again $\int_0^t\|\Gamma(\tau, \cdot)\|_{W^{1,1}(\R)}d\tau\leq C(1-e^{-t}+t^{\frac{1}{2}})$ from \eqref{Gamma}. Then we get from (\ref{leqn5n*})-(\ref{leqn7*})
\begin{eqnarray}\label{leqn8*}
\begin{aligned}
&\|S-\bar{S}\|_{\mathcal{Y}}
&\leq Cm(t)(\|\theta-\bar{\theta}\|_{\mathcal{X}}+\|S-\bar{S}\|_{\mathcal{Y}})
\end{aligned}
\end{eqnarray}
where $m(t)=1-e^{-t}+t^{\frac{1}{2}}$ as above and the inequality $\|\rho-\bar{\varrho}\|_{L^{\infty}(\R)}\leq C\|\theta-\bar{\theta}\|_{\mathcal{X}}$ has been used. Taking $T>0$ small enough with $0<t<T$ such that $C m(t)<1$, we have from \eqref{leqn8*} that
\begin{eqnarray}\label{leqn11*}
\begin{aligned}
\|S-\bar{S}\|_{\mathcal{Y}} \leq \frac{Cm(t)}{1-C m(t)}\|f-\tilde{f}\|_{\mathcal{X}}=: \ell_1(t)\|\theta-\bar{\theta}\|_{\mathcal{X}}.
\end{aligned}
\end{eqnarray}

{\it Step 2}. Let $S \in \mathcal{Y}(T)$ be the fixed point of \eqref{leqn5} satisfying \eqref{pes} obtained in Step 1. We consider the following problem
\begin{equation}\label{lenq1}
f_t+ v\cdot \nabla_x f=\int_V (T[S]f'-T^*[S]f)dv', \ \ f|_{t=0}=f_0.
\end{equation}
Using the backward characteristic starting at $(t,x)$, the characteristic curve for any $0 \leq \tau \leq  t$ is given by
\begin{equation*}
{X}(\tau;t,x)=x-v(t-\tau).
\end{equation*}
Then we can rewrite (\ref{lenq1}) along the characteristic curve as
\begin{eqnarray*}\label{leqn2}
\begin{aligned}
f(t,x,v)=f_0(X(0),v)+\int_0^t \int_V&\{T[S(\tau, X(\tau))]f(\tau, X(\tau),v')\\
& \ \ \  -T^*[S(\tau, X(\tau))]f(\tau, X(\tau),v)\}dv'd\tau.
\end{aligned}
\end{eqnarray*}
Next we define an operator $\mathcal{T}_2:\mathcal{X} \to \mathcal{X}$ such that
\begin{equation*}
\mathcal{T}_2[f]=f_0(X(0))+\int_0^t \int_V(T[w(\tau, X(\tau))]f(\tau, X(\tau),v')-T^*[w(\tau, X(\tau))]f(\tau, X(\tau),v))dv'd\tau.
\end{equation*}
Then for any $f, \tilde{f} \in {  \mathcal{X}(T)}$ with $f_0=\tilde{f}_0$, we have
\begin{eqnarray*}
\begin{aligned}
\mathcal{T}_2[f]-\mathcal{T}_2[\tilde{f}]&=\int_0^t \int_V T[S(\tau, X(\tau))](f(\tau, X(\tau),v')-\tilde{f}(\tau, X(\tau),v'))dv'd\tau\\
& \ \ -\int_0^t \int_V T^*[S(\tau, X(\tau))](f(\tau, X(\tau),v)-\tilde{f}(\tau, X(\tau),v))dv'd\tau.
\end{aligned}
\end{eqnarray*}
By the hypothesis (H) along with \eqref{pes}, one deduces for any $0<t\leq T$ that
\begin{eqnarray*}\label{leqn3}
\begin{aligned}
\|\mathcal{T}_2[f]-\mathcal{T}_2[\tilde{f}]\|_{\mathcal{X}}&\leq 2C_0(1+\sup\limits_{0\leq \tau \leq T} \|S(\tau)\|_{W^{1,\infty}(\R)}+\|S_{t}(t)\|_{L^{\infty}(\mathbb{R})}) \int_0^t\int_V \|f-\tilde{f}\|_{\mathcal{X}}dv'd\tau\\
%& \leq 2C_0|V|t(1+t+\sup\limits_{0\leq \tau \leq T} \|S(\tau)\|_{W^{1,\infty}(\R)}) \|f-\tilde{f}\|_{\mathcal{X}}\\
& \leq 2C|V|t(1+t) \|f-\tilde{f}\|_{\mathcal{X}},
\end{aligned}
\end{eqnarray*}
where Lemma \ref{1destimate1n} has been used. Now let $T>0$ small with $0<t<T$ such that
$
0<t(1+t)<\frac{1}{2C|V|}.
$
Then $\mathcal{T}_2$ is a contraction mapping on $\mathcal{X}(T)$ and hence has a unique fixed point $f \in \mathcal{X}(T)$.

Now we investigate the continuous dependence of  $f$ on $S$. To this end, we replace $S$ by $\tilde{S}\in \mathcal{Y}(T)$ in (\ref{lenq1}) with $S_0=\tilde{S}_0$, which gives another unique solution $\tilde{f}\in \mathcal{X}(T)$ by the above argument. Then
\begin{eqnarray*}
\begin{aligned}
f-\tilde{f}=&\int_0^t\int_V [(T[S]-T[\tilde{S}])f'+(T^*[\tilde{S}]-T^*[S])f]dv'd\tau\\
&+\int_0^t\int_V(T[\tilde{S}](f'-\tilde{f}')-T^*[\tilde{S}](f-\tilde{f}))dv'd\tau.
\end{aligned}
\end{eqnarray*}
From the assumption (H) and \eqref{pes}, it follows that
\begin{eqnarray*}
\begin{aligned}
\|f-\tilde{f}\|_{\mathcal{X}} \leq 2L|V| t  (\|f\|_{\mathcal{X}}+\|\tilde{f}\|_{\mathcal{X}})  \|S-\tilde{S}\|_{\mathcal{Y}}+2C|V|t(1+t+\|S\|_{\mathcal{Y}}+\|\tilde{S}\|_{\mathcal{Y}})\|f-\tilde{f}\|_{\mathcal{X}},
\end{aligned}
\end{eqnarray*}
where $L$ is a Lipschitz constant. %, $N_T=\|f\|_{\mathcal{X}}$ and $\tilde{N}_T=\|\tilde{w}\|_{\mathcal{Y}}$.
If we let $T>0$ small enough with $0<t<T$ such that $t(1+t+\|S\|_{\mathcal{Y}}+\|\tilde{S}\|_{\mathcal{Y}}) <\frac{1}{2|V|}$, then
\begin{equation}\label{leqn9}
\|f-\tilde{f}\|_{\mathcal{X}} \leq \frac{2L|V|t(\|f\|_{\mathcal{X}}+\|\tilde{f}\|_{\mathcal{X}})}{1-2|V|t (1+t+\|S\|_{\mathcal{Y}}+\|\tilde{S}\|_{\mathcal{Y}})}\|w-\tilde{w}\|_{\mathcal{Y}}=: \ell_2(t)\|S-\tilde{S}\|_{\mathcal{Y}}.
\end{equation}

{\it Step 3}. From the procedures shown above, we get a map $\phi_1: \mathcal{X}(T) \to \mathcal{Y}(T)$ in Step 1 and another map $\phi_2:\mathcal{Y}(T) \to \mathcal{X}(T)$ for small $T>0$ in Step 2. Now we define a map $\Phi=\phi_2 \circ \phi_1: \mathcal{X}(T) \to \mathcal{X}(T)$ such that $\Phi(\theta)=f$.  Given two functions $\theta_1, \theta_2 \in \mathcal{X}$, we denote $S_i=\phi_1(\theta_i)$ and $f_1=\phi_2(S_i)$ for $i=1,2$. Then using  \eqref{leqn9} and  \eqref{leqn11*}, we get
% we have
%$$\|S_1-S_2\|_{\mathcal{Y}}\leq \ell_1(t)\|\theta-\bar{\theta}\|_{\mathcal{X}}.$$
%, it follows that
%$$\|f_1-f_2\|_{\mathcal{X}}\leq \ell_2(t)\|S_1-S_2\|_{\mathcal{Y}}.$$
%Therefore, one gets
\begin{eqnarray*}
\begin{aligned}
\|\Phi(\theta_1)- \Phi(\theta_2)\|_{\mathcal{X}} \leq \|\phi_2(S_1)-\phi_2(S_2)\|_{\mathcal{X}} &\leq \|f_1-f_2\|_{\mathcal{X}}\leq \ell_1(t) \ell_2(t)\|\theta_1-\theta_2\|_{\mathcal{X}}
\end{aligned}
\end{eqnarray*}
where $\ell_1(t)$ are $\ell_2(t)$ can be made arbitrarily small if $t$ is small enough. Hence for {  $T>0$} sufficiently small with $0<t<T$, $\Phi$ is a contraction mapping and hence has a unique fixed point  which gives to a unique solution $f\in \mathcal{X}(T)$. Furthermore we can get $f\in L^1(\R \times V)$ directly by integrating \eqref{eq2a7}. Using this $f$, we get a unique solution $S\in \mathcal{Y}(T)$ from the $S$-equation as shown in Step 1. To finish the proof, it remains only to show $S_{xx}\in L^\infty([0,T); L^\infty(\R))$. Indeed from Lemma \ref{1destimate2}, one has $S_t \in L^\infty([0,T); L^\infty(\R))$. Note that \eqref{eq2a9} gives  $S_{xx}=S_t-g(S)+z\rho+S(1+S)$, which along with the fact that $S_t, g, S$ and $z\rho$ are bounded for any $t\in [0, T)$ shown above entails that $S_{xx}\in L^\infty([0,T); L^\infty(\R))$. This completes the proof of Lemma \ref{local}.
\end{proof}

\subsection{Proof of Theorem \ref{mth}}
By the continuity argument, to obtain the global existence, it suffices to
derive {\it a priori} bound for the solution. Note that
(\ref{eq4a35}) is a direct consequence of (\ref{eq3a14}). Then it remains to show
(\ref{eq4a34}) and (\ref{eq4a36}). To this end, we integrate (\ref{eq2a7}) along the
characteristic curve and use the assumption (\ref{eq2a6}) along with the hypothesis (H) to
have that
\begin{eqnarray*}
\begin{array}{lll}
f(t,x,v)&=&f_0(x-vt,v)-\D \int_0^t\int_V T^*[S(t-\tau,x-v\tau)]f(t-\tau,x-v\tau,v)dv'd\tau\\
&& \ \ \ \ \  \D+\int_0^t\int_V T[S(t-\tau,x-v\tau)]{  f}(t-\tau,x-v\tau,v')dv'd\tau \\
&\leq& f_0(x-vt,v)+\D\int_0^t\int_V T[S(t-\tau,x-v\tau)](t-\tau,x-v\tau,v')dv'd\tau.
%&\leq& f_0(x-vt,v) +C \D \int_0^t
%\bigg\{\Big(1+|S(t-\tau,x-v\tau)|\\
%&& \ \ \ \ \  \D+|\partial_x S(t-\tau,x-v\tau)|\Big)\int_V
%f(t-\tau,x-v\tau,v')dv'\bigg\}d\tau.
\end{array}
\end{eqnarray*}
Then using the hypothesis (H), we have
\begin{eqnarray}\label{eq4a37}
\begin{array}{lll}
f(t,x,v)&\leq& f_0(x-vt,v) +C \D \int_0^t
\bigg\{\Big(1+|S(t-\tau,x-v\tau)|+|S_x(t-\tau,x-v\tau)|\\
&& \ \ \ \ \  \D+|S_t(t-\tau,x-v\tau)|\Big)\int_V
f(t-\tau,x-v\tau,v')dv'\bigg\}d\tau.
\end{array}
\end{eqnarray}
We take the $p$-th power of (\ref{eq4a37}) and integrate the result with
respect to $x$ and $v$. Then using the compactness of $V$, we end up with
\begin{eqnarray}\label{eq4a38}
\begin{array}{lll}
\|f(t)\|_{L^p(\R \times V)}  &\leq& \D \|f_0\|_{L^p(\R \times V)}\\
&&\ \ \ \  +C
\D\int_0^t (1+\|S(\tau)\|_{W^{1,{\infty}}(\R \times V)}+\|S_t\|_{L^\infty(\R)})
\|f(\tau)\|_{L^p(\R \times V)} d\tau.
\end{array}
\end{eqnarray}
Using Lemma \ref{1destimate1} and Lemma \ref{1destimate2}, and taking $p=2$, we get for all $t \geq 0$
\begin{eqnarray}\label{eq4a39}
\begin{aligned}
\|f(t)\|_{L^2(\R \times V)}  &\leq \D \|f_0\|_{L^2(\R \times V)} +C
\int_0^t \|f(\tau)\|_{L^2(\R \times V)}d\tau \\
&\D \ \ \  +C \int_0^t (1+(\ln \tau)_++ (\ln \sup\limits_{0\leq s \leq \tau} \|\rho(s)\|_{L^2(\R
\times V)})_+)\cdot\|f(\tau)\|_{L^2(\R \times V)} d\tau\\
& \leq \D \|f_0\|_{L^2(\R \times V)} +C
\int_0^t\|f(\tau)\|_{L^2(\R \times V)}d\tau \\
&\D \ \ \ +C \int_0^t a(\tau) (\ln \sup\limits_{0\leq s \leq \tau} \|f(s)\|_{L^2(\R
\times V)})_+\|f(\tau)\|_{L^2(\R \times V)} d\tau
\end{aligned}
\end{eqnarray}
where the fact $\|\rho(t)\|_{L^2(\R \times V)} \leq C(V) \|f(t)\|_{L^2(\R
\times V)}$ has been used and $a(\tau)=1+(\ln \tau)_+$.

Setting
\begin{equation*}
y(\tau)= \sup\limits_{0\leq s \leq \tau} \|f(s)\|_{L^2(\R
\times V)}
\end{equation*}
we obtain from (\ref{eq4a39}) that
\begin{eqnarray*}
\begin{aligned}
y(t) \leq &\|f_0\|_{L^2(\R \times V)} +C \int_0^t [a(\tau) y(\tau)(\ln y(\tau))_++y(\tau)]d\tau.
\end{aligned}
\end{eqnarray*}
Then applying the Gronwall's inequality in Lemma \ref{grownwall}, we obtain
\begin{eqnarray}\label{eq4a40}
\begin{array}{lll}
\|f(t)\|_{L^2(\R \times V)} \leq C_1(t)
\end{array}
\end{eqnarray}
where $C_1(t)=\big[(1+\|f_0\|_{L^2(\R \times V)}) \exp\big(\int_0^t [1+a(\tau)]
e^{-\int_0^{\tau}a(s)ds}d\tau\big)\big]^{\exp(\int_0^t a(s)ds)}$ is bounded for any $0<t\leq T<\infty$. It is evident
from (\ref{eq4a2}) and (\ref{eq4a40}) that there is another constant $C_2(t)$ bounded for any $0<t \leq T<\infty$ such that
\begin{eqnarray}\label{eq4a41}
\begin{array}{lll}
\|S_x(t)\|_{L^{\infty}(\R)} \leq
C_2(t)
\end{array}
\end{eqnarray}
which, along with Lemma \ref{1destimate1n}, indicates that $S\in L^{\infty}([0, T]; W^{1,\infty}(\R))$ for any $0<t\leq T<\infty$. Then  applying (\ref{eq4a40}) and (\ref{eq4a41}) into (\ref{eq4a38}) with the Gronwall's inequality, we obtain the following inequality
\begin{eqnarray}\label{eq4a42}
\begin{array}{lll}
\|f(t)\|_{L^{p}(\R \times V)} \leq C_3(t), \ \ 1 \leq
p\leq \infty
\end{array}
\end{eqnarray}
which implies (\ref{eq4a34}). Furthermore the application of Lemma \ref{1destimate2} indicates that
$$S_t \in L^{\infty}([0, T]; L^{\infty}(\R)).$$
By the standard argument of temporal regularity for parabolic equations (e.g., see \cite{RJC}), we have $S \in C([0,T];L^{\infty}(\R))$.
%By Sobolev embedding theorem $W^{1,\infty}(\R)) \hookrightarrow C^{0,1}(\R)$, one has $S\in C([0,\infty)\times \R)$.

Next we derive $ S_x \in L^p(\R)$ for any $1 \leq p< \infty$ and any $t>0$. In fact from (\ref{eq4a5}), it holds that
\begin{eqnarray*}\label{eq4a44}
\begin{array}{lll}
\|S_x(t)\|_{L^p(\R)} &\leq& \D \| \Gamma_x(t, \cdot) \ast S_0\|_{L^p(\R)}+\int_0^t \|\Gamma_x(s)
\ast (g(S)-z) \rho\|_{L^p(\R)}ds\\
&&\D+\int_0^t \|\Gamma_x(s) \ast
S^2\|_{L^p(\R)}ds\\
&\leq& \D \|\Gamma_x(t, \cdot)\|_{L^1(\R)} \|S_0\|_{L^p(\R)}+
\int_0^t \|\Gamma_x(s)\|_{L^p(\R)}
\|\rho\|_{L^1(\R)}ds\\
&&\D+\int_0^t \|\Gamma_x(s)\|_{L^p(\R)}
\|S\|_{L^2(\R)}^2ds.
\end{array}
\end{eqnarray*}
Then $\|S_x(t)\|_{L^p(\R)}<\infty$ for $1 \leq p< \infty$ and any $t>0$
due to Lemma \ref{fundsolu1}, (\ref{eq3a10}) and (\ref{eq4a1}).

To finish the proof, it remains to derive the $L^{\infty}$ estimates for $\partial_{x}^2S$. Indeed using the equation (\ref{eq2a9}), the second spatial
derivative of external signal can be expressed as
\begin{equation*}\label{eq4a43}
S_{xx}={\cred S_t}-g(S)+{\cred z\rho}+S(1+S).
\end{equation*}
Then it is easy to see that $S_{xx} \in L^{\infty}([0, T]; L^{\infty}(\R))$ for any $0<T<\infty$ by (\ref{eq4a1}), (\ref{eq4a3}), (\ref{eq4a40}) and
(\ref{eq4a42}). Combining the above results with Lemma \ref{estimatez}, we finish the proof of Theorem \ref{mth}.

\qed

\section{Hydrodynamic Limits}
\setcounter{equation}{0}
\renewcommand{\theequation}{\thesection.\arabic{equation}}
The model (\ref{eq2a7})-(\ref{eq2a9}) studied in the present paper includes
the internal state $z$, which {\cred impacts} the turning kernel $T[S]$ implicitly
through {\cred affecting} the chemical signal $S$.  In this section, we shall extend the hydrodynamic limit of the model
to a more general turning kernel than that in \cite{DS}.

We start with reformulating the equation (\ref{eq2a7}) as follows
\begin{equation}\label{eq5a1}
\partial_t f+v\cdot \nabla_x f=Q(f),
\end{equation}
where
\begin{eqnarray*}\label{operator}
 Q(f)=\D-\lambda[S](v)f(v)+\int_VT[S]f(v')dv', \ \lambda[S](v)=\D\int_V T^*[S](v',v)dv'
\end{eqnarray*}
where we have used the abbreviated notation $f(v):=f(t,x,v)$ and $f(v'):=f(t,x,v')$.

We assume that the turning rate $\lambda[S](v)$ has a lower bound $\lambda_1$ and an upper bound $\lambda_2$:
\begin{equation}\label{assumptions3}
0<\lambda_1\leq \lambda[S](v) \leq \lambda_2.
\end{equation}
We remark  that the assumption (\ref{assumptions3}) is more general than (\ref{eq1a18}) made in \cite{DS}.
Particularly in our assumption, the turning kernel $T[S]$ can be a Dirac delta function of velocity  $v$ which is unbounded.
This extension is of importance in applications. For instance, in case of mesenchymal
motion \cite{hilmesenchymal, HWH}, cell motion is highly guided by fibre
orientation, and the fibre distribution is a Dirac delta distribution when fibres are totally aligned. For bacteria motion using a ``run-and-tumble'' strategy based on a
velocity-jump process, the jumps are
instantaneous and consequently the turning kernel $T$ {  can be} given by a Dirac
distribution (see \cite{EO3}). In such scenario,  measurable solutions may be considered  \cite{HWH}. However the aim of this section is to carry out the hyperbolic limits under the generalized assumption (\ref{assumptions3}) with an idea of \cite{Poupaud1} where the crucial element is to study the invertibility of the operator $Q$ in a suitable space $L_{\lambda}^1(V)$ - a weighted $L^1$ space defined by
\begin{equation*}
L_{\lambda}^1(V)=\bigg\{f: \int_V|f(v)|\lambda[S](v)dv<\infty \bigg\}.
\end{equation*}
Noticing that $L^1(V)=L_{\lambda}^1(V)$ under the assumption (\ref{assumptions3}), we have the following results for the null space $N(Q)$ of the operator $Q$ (see Theorem 1 in \cite{Poupaud1}).
\begin{lemma}\label{null}
The following two conclusions have only one true

(1) $N(Q)=\{0\}$;

(2) There is a unique positive function $F(v)$ such that $Q(F)=0, \ \int_V F(v)dv=1$ and $N(Q)=\{\alpha F: \alpha \in \R \}$.

\end{lemma}

%{\it Proof}. The difference of our Lemma from Theorem 1 in \cite{Poupaud1} is that the velocity space is compact in the present paper whereas is unbounded in \cite{Poupaud1}. But the arguments used there can be adopted to a compact set directly. So we omit the details of the proof and refer interested readers to \cite{Poupaud1}.

\qed

The second conclusion of Lemma \ref{null} implies that the kernel of
the turning operator $Q$ is one dimensional and spanned by the
equilibrium distribution $F(v)$. Furthermore, we have the following
result for the invertibility of $Q$ (see \cite[Theorem 2]{Poupaud1}).

\begin{lemma}\label{inverse}
For any function $\phi \in L^1(V)$, the problem $Q(f)=\phi$ has a solution if and only if $\int_V \phi(v){  dv}=0$ and the solution is unique in $L^{1,0}(V)$, where
\begin{equation*}
L^{1,0}(V)=\bigg\{f \in L^1(V): \int_Vf(v)dv=0\bigg\}.
\end{equation*}
\end{lemma}

%{\it Proof}. The proof is similar to the proof of Theorem 2 in \cite{Poupaud1}. So we omit the details.

%\qed

We now substitute  the hyperbolic scaling
$\bar{x}=\varepsilon x,\ \bar{t}=\varepsilon t$
into (\ref{eq5a1}), where $\varepsilon$ is a small parameter. Dropping the bars for  convenience, we obtain
\begin{equation}\label{hyperbolic}
\varepsilon \partial_t f_{\varepsilon}+\varepsilon v \cdot \nabla_x f_{\varepsilon}=Q(f_{\varepsilon}),
\end{equation}
with initial data
\begin{equation}\label{initial}
f_{\varepsilon}(0,x,v)=f_I(x,v).
\end{equation}
For simplicity, we assume here that $Q(f_I(x,v))=0$ to avoid a problem of initial layers (see \cite{DS} for the discussion on the initial layer problem).

We expand $f_{\varepsilon}$ in terms of $\varepsilon: f_{\varepsilon}=f^0+\varepsilon f^1+\cdots$ satisfying
\begin{equation*}
f^0(t=0)=f_I.
\end{equation*}
Substituting this expansion into (\ref{hyperbolic}) and equating the same order term of $\varepsilon$, we find
\begin{eqnarray}
\varepsilon^0 &:& Q(f^0)=0\label{zero}\\
\varepsilon^1 &:& \partial_t f^0+v\cdot \nabla_x f^0=Q(f^1). \label{first}
%\varepsilon^2 &:& \partial_t f^1+v\cdot \nabla_x f^1=Q(f_2) \label{second}
\end{eqnarray}
Note that we look for nonzero leading order term. Then from (\ref{zero}) and Lemma \ref{null}, we deduce that $f^0(t,x,v)=\rho^0(t,x)F(t,x,v)$ with $\rho^0=\int_V f^0(v)dv$. By equation (\ref{first}) and Lemma \ref{inverse}, we get
\begin{equation}\label{limit}
\frac{\partial \rho^0}{\partial t}+\nabla \cdot(\sigma \rho^0)=0, \ \ \sigma=\int_VvF(v)dv
\end{equation}
satisfying the initial condition
\begin{equation}\label{initial2}
\rho^0(t=0)=\int_V f_I(v)dv.
\end{equation}

Then we have the following result analogous to the one in \cite{DS} but with weaker assumption than \eqref{eq1a18} made in \cite{DS}.
\begin{thm}
Let assumption (\ref{assumptions3}) hold. Let $f_{\varepsilon}$ and $\rho^0$ be the solutions of problem (\ref{hyperbolic})-(\ref{initial}) and (\ref{limit})-(\ref{initial2}) for any $(x,v)\in \R^N\times V (N\geq 1)$, respectively. Let $F$ be the equilibrium distribution spanning the kernel of the turning operator $Q$. Then $f_{\varepsilon}(t,x,v) \to \rho^0(t,x)F(t,x,v)$ as $\varepsilon \to 0$ such that for any $0\leq t \leq T$ it holds that
\begin{equation*}
\|f_{\varepsilon}(t,\cdot,\cdot)- \rho^0(t,\cdot)F(t,\cdot,\cdot)\|_{L^1(\R^N\times V)} \leq C_T \varepsilon,
\end{equation*}
where $C_T$ is a constant depending on $T$.
\end{thm}

{\it Proof}. The proof is in the same spirit of  \cite{Poupaud1} and hence will be sketched only. We first derive from equations (\ref{limit}) and (\ref{initial2}) that
\begin{equation*}
\partial_t f^0 +v \cdot \nabla_x f^0 \in L^{1,0}(V).
\end{equation*}
Noting that $f_I=\rho^0(t=0)F(t=0)$ and $f^0(t=0)=f_I$, we derive that the residue $r_{\varepsilon}=f_{\varepsilon}-f^0-\varepsilon f^1$ is the solution of the problem
\begin{eqnarray}\label{remainder}
\begin{aligned}
\partial_t r_{\varepsilon}+v \cdot \nabla_x r_{\varepsilon}&=\frac{1}{\varepsilon} Q(f_{\varepsilon})-\varepsilon (\partial_t f^1+v\cdot \nabla_x f^1),\\
r_{\varepsilon}(0)&=\varepsilon f^1(0).
\end{aligned}
\end{eqnarray}
%where the higher order terms are ignored.
Integrating (\ref{remainder}) along the characteristic curve, we
get
\begin{eqnarray*}
\begin{aligned}
r_{\varepsilon}(t,x,v) = & r_{\varepsilon}(0, x-vt,v)-\frac{1}{\varepsilon}\int_0^t Q(f_{\varepsilon}(t-s,x-vs,v))ds\\
&-\varepsilon \int_0^t (\partial_t f^1(t-s,x-vs,v)+v \cdot \nabla_xf^1(t-s,x-vs,v))ds.
\end{aligned}
\end{eqnarray*}
Since $V$ is compact in $\R^N$ and $\int_V Q(f_{\varepsilon})dv=0$, we derive from above inequality that
\begin{eqnarray*}
\begin{aligned}
\|r_{\varepsilon}(t,\cdot,\cdot)\|_{L^1(\R^N\times V)} = &\varepsilon \|f^1(0,\cdot,\cdot)\|_{L^1(\R^N\times V)}\\
&+\varepsilon \int_0^t \|\partial_t f^1+v \cdot \nabla_xf^1(t-s,\cdot,\cdot)\|_{L^1(\R^N\times V)}ds.
\end{aligned}
\end{eqnarray*}
It then follows that
\begin{eqnarray*}
\|r_{\varepsilon}(t,\cdot,\cdot)-f^0(t,\cdot,\cdot)\|_{L^1(\R^N\times V)} \leq (C_1+C_2T)\varepsilon,
\end{eqnarray*}
where
\begin{eqnarray*}
\begin{aligned}
C_1&=\|f^1(0,\cdot,\cdot)\|_{L^1(\R^N\times V)}+\sup\limits_{0\leq t\leq T}\|f^1(t,\cdot,\cdot)\|_{L^1(\R^N\times V)},\\
C_2&= \sup\limits_{0\leq t\leq T} \|\partial_t f^1+v \cdot \nabla_xf^1(t-s,\cdot,\cdot)\|_{L^1(\R^N\times V)}.
\end{aligned}
\end{eqnarray*}
To finish the proof, it remains to show that
\begin{equation*}
f^1, \partial_t f^1+v\cdot \nabla_x f^1\in L^{\infty}([0,T); L^1(\R^N\times V)).
\end{equation*}
which was shown in \cite[section 5]{Poupaud1}. Hence we omit the details and complete the proof.
\qed

\bigbreak
\noindent \textbf{Acknowledgment.}
%The author is very thankful to the referees for important comments and suggestions which greatly improve the exposition of this paper. 
This research is supported by the Hong Kong Research Grant Council General Research Fund No. PolyU 153055/18P (Project ID: P0005472).

\end{document}